\documentclass[]{sn-jnl}% Math and Physical Sciences Numbered Reference Style 
\usepackage{graphicx}%
\usepackage{multirow}%
\usepackage{amsmath,amssymb,amsfonts}%
\usepackage{amsthm}%
\usepackage{mathrsfs}%
\usepackage[title]{appendix}%
\usepackage{xcolor}%
\usepackage{textcomp}%
\usepackage{manyfoot}%
\usepackage{booktabs}%
\usepackage{algorithm}%
\usepackage{algorithmicx}%
\usepackage{algpseudocode}%
\usepackage{listings,enumerate,cite}%
\theoremstyle{thmstyleone}%
\newtheorem{theorem}{Theorem}%  meant for continuous numbers
\newtheorem{proposition}{Proposition}
\newtheorem{lemma}{Lemma}
\theoremstyle{thmstyletwo}%
\newtheorem{corollary}{Corollary}

\theoremstyle{definition} % Non-italic style for theorems, remarks, etc.
\newtheorem{remark}[theorem]{Remark}
\newtheorem{example}[theorem]{Example}

\theoremstyle{thmstylethree}%
\newtheorem{definition}{Definition}%
\newcommand{\vertiii}[1]{\left\vert\kern-0.25ex\left\vert\kern-0.25ex\left\vert #1\right\vert\kern-0.25ex\right\vert\kern-0.25ex\right\vert}
\newcommand {\dom} {{\rm dom}\,}
\newcommand {\epi} {{\rm epi}\,}
\newcommand{\al}{\alpha}
\newcommand{\be}{\beta}
\newcommand{\ga}{\gamma}
\newcommand{\la}{\lambda}

\newcommand{\eps}{\varepsilon}
\newcommand{\bx}{\bar x}

\newcommand {\R} {\mathbb R}
\newcommand {\N} {\mathbb N}

\newcommand{\norm}[1]{\left\Vert#1\right\Vert}
\newcommand{\abs}[1]{\left\vert#1\right\vert}
\newcommand {\sd} {\partial}
\newcommand{\by}{\bar y}
\newcommand {\B} {\mathbb B}

\usepackage{mathptmx}   %Times New Roman
\hypersetup{
	colorlinks=true,
	linkcolor=blue, % Couleur des liens internes
	citecolor=red, % Couleur des num?ros de la biblio dans le corps
	urlcolor=blue  } % Couleur des url
\begin{document}

\title[Primal and dual characterizations of sign-symmetric norms]{Primal and dual characterizations of sign-symmetric norms}

\author*{\fnm{Nguyen} \sur{Duy Cuong}}
\email{ndcuong@ctu.edu.vn}

\affil{\orgdiv{Department of Mathematics}, \orgname{College of Natural Sciences, Can Tho University},  \city{Can Tho City}, \country{Vietnam}}

\abstract{The paper studies primal and dual characterizations of a class of sign-symmetric norms on product vector spaces. Correspondences between these norms and a class of convex functions are established. 
Explicit formulas for the dual norm and the convex subdifferential of a given primal norm are derived. 
It is demonstrated that this class of norms is well-suited for studying properties and problems on product spaces.
As an application, we study the von Neumann-Jordan constant of norms on product spaces and extend a classical result of Clarkson from Lebesgue spaces to general normed vector spaces.}



\maketitle

\section{Introduction}
The paper is an attempt to systematically study \textit{sign-symmetric} norms on products of vector spaces.
Many important properties and problems in variational analysis involve the use of norms on product spaces.
These include, for example,  extremality and transversality properties of collections of sets, metric regularity properties of set-valued mappings, and optimization problems  on product spaces.
For a broader range of topics involving product structures, the reader is referred to the monographs \cite{RocWet98,Mor06.1,DonRoc14,Iof17}.

The most commonly used norms on product vector spaces are  the conventional maximum and sum norms, and in some cases, the $p$-norms.
The product norm construction machinery proposed in this paper provides flexibility in selecting norms on product spaces to study  properties and problems in variational analysis.
On the other hand,  the construction of norms on product spaces is itself an important topic of study.
The reader is referred to \cite{MitSaiSuz03, MitOshSai05, SomAttSat05} and the references therein for various studies related to norms on products of complex spaces.

Let $X$ be a normed space, and $n\ge 2$.
A norm $\vertiii{\cdot}$ on $X^n$ is said to be
\textit{sign-symmetric} if
	\begin{gather}\label{A1}\tag{{A1}}
		\vertiii{(x_1,\ldots,x_n)}=\vertiii{(\pm x_1,\ldots,\pm x_n)}\;\;\text{for all}\;\;(x_1,\ldots,x_n)\in X^n.
	\end{gather}
Observe that all $p$-norms are sign-symmetric. 
In this paper, we demonstrate that this class of norms provides a flexible framework for study properties and problems  on product spaces.
Several primal and dual characterizations of such norms are established, contributing to a deeper understanding of their structures in primal and dual spaces.

The concept of sign-symmetric norms on general vector spaces  can be seen as
a generalization of the notion of \textit{absolute} norms, originally introduced by Bonsall and Duncan \cite{BonDun73} for the complex space $\mathbb{C}^2$, and later extended to $\mathbb{C}^n$ by Saito, Kato, and Takahashi \cite{SaiKatTak00}.
Building on techniques developed in \cite{SaiKatTak00}, we establish correspondences between the class of sign-symmetric (strictly convex) norms on general product vector spaces and a class of (strictly) convex continuous functions on the standard simplex in $\R^n$.

Explicit formulas for the dual norm and the convex subdifferential of a  sign-symmetric primal norm are provided.
It is demonstrated that when a primal norm is associated with a convex continuous function, its dual corresponds to a convex continuous conjugate-type function of the primal function. 
Duality relations between these functions are established, leading to corresponding relations between the  primal and dual norms. 
We show that any pair of primal norms as well as their corresponding dual norms are topologically equivalent.
On the other hand, the subdifferential formula is derived through a two-step procedure. 
First, we compute the subdifferential of the norm at a vector with real-valued components.
The formula is then used to compute the subdifferential of the norm at an arbitrary point in product  spaces. 

We demonstrate that the class of sign-symmetric norms studied in this paper is well-suited for generalized separation results for collections of sets.
For many years, $p$-norms have been used when formulating  separation results.
In the recent study \cite{CuoKru24-2}, the authors introduce a set of quantitative relations between norms on product spaces that are sufficient for establishing such results.
We show that sign-symmetric norms satisfy all the required conditions.

It is well known that the norm induced by an inner product on a vector space satisfies the parallelogram identity.
The von Neumann–Jordan constant of a norm  measures how closely this norm can be characterized by an inner product \cite{JorNeu35}.
In this paper, we derive estimates (precise in some cases) of the von Neumann–Jordan constant of sign-symmetric norms and extend a classical result of Clarkson \cite{Cla37} from Lebesgue spaces to general vector spaces.

The next  Section~\ref{S3} establishes relations between  
the  class of sign-symmetric (strictly convex) norms  and a class of (strictly) convex continuous functions.
Sections \ref{S4} and \ref{S5} are dedicated to dual norm and subdifferential characterizations of sign-symmetric norms.
We demonstrate in Section~\ref{S6} that the class of sign-symmetric norms is particularly well-suited for formulating and proving generalized separation results for collections of sets.
As an application, we study in Section~\ref{S7} the von Neumann-Jordan constant of sign-symmetric norms and revisit a  result of Clarkson.
The final Section~\ref{S8} summarizes the main contributions of the paper and outlines potential directions for future research.

\subsubsection*{Preliminaries}
Our basic notation is standard, see, e.g., \cite{Mor06.1,Iof17,RocWet98,DonRoc14} .
Throughout the paper, if not explicitly stated otherwise, $X$ is a vector space equipped with a norm $\|\cdot\|$.
Given an integer $n\ge 2$, the product space $X^n$ is assumed to be endowed with a norm $\vertiii{\cdot}$.
The topological dual of a vector space $X$ is denoted by $X^*$, while $\langle\cdot,\cdot\rangle$ denotes the bilinear form defining the pairing between the two spaces.
We keep the same notations $\|\cdot\|$ and $\vertiii{\cdot}$ (possibly with a subscript indicating the space)  for the corresponding dual norms on $X^*$ and $(X^n)^*$.
We write $0_X$ to emphasize that the zero vector belongs to the vector space $X$.
The notation $\mathbb{S}^n_X:=\mathbb{S}_X\times\ldots\times \mathbb{S}_X$ stands for the Cartesian product of $n$ unit spheres.
Symbols $\B_{X}$ and $B_\rho(\bx)$ denote the open unit ball and open ball with centre $\bx$ and radius $\rho$, respectively, while $\overline\B_X$ and $\overline B_\rho(\bx)$  denote the corresponding closed balls.
We write $\R$ and $\N$ to denote the sets of all real numbers and all positive integers, respectively.
The notation $\infty$ represents $+\infty$ in the extended real line with the convention that $\frac{1}{\infty}=0$, while
the notation
$\{x^k\}\subset\Omega$ refers to a sequence of points $x^k\in\Omega$ for all $k\in\N$.
The sign function is defined by
$\text{sgn}(x)=1$ if $x>0$, $\text{sgn}(x)=0$ if $x=0$, and
$\text{sgn}(x)=-1$ if $x<0$.

Given a subset $\Omega$ of a normed space $X$ and a point $\bx\in \Omega$, the set
	\begin{gather*}%\label{NC}
		N_{\Omega}^F(\bx):= \Big\{x^\ast\in X^\ast\mid
		\limsup_{\Omega\ni x{\rightarrow}\bar x,\;x\ne\bx} \frac {\langle x^\ast,x-\bx\rangle}
		{\|x-\bx\|} \le 0 \Big\}
	\end{gather*}
is the Fr\'echet normal cone to $\Omega$ at $\bx$.
If $\Omega$ is a convex set, it reduces to the normal cone in the sense of convex analysis:
$N_{\Omega}(\bx):= \left\{x^*\in X^*\mid \langle x^*,x-\bx \rangle \leq 0\;\;\text{for all}\;\;	x\in \Omega\right\}.$
For a function $f:X\to\R\cup\{\infty\}$ on a normed space $X$,
its domain and epigraph are defined,
respectively, by
$\dom f:=\{x \in X\mid f(x) <\infty\}$
and
$\epi f:=\{(x,\alpha) \in X \times \mathbb{R}\mid f(x) \le \alpha\}$.
The Fr\'echet subdifferential of $f$ at $\bar x\in\dom f$
is defined by
	\begin{gather*}
		\sd^F f(\bar x):= \left\{x^* \in X^*\mid (x^*,-1) \in N^F_{\epi f}(\bar x,f(\bar x))\right\}.
	\end{gather*}
If $f$ is convex, it
reduces to the subdifferential in the sense of convex analysis.
In this case, we simply write $\partial{f}(\bx)$.
We set $N_{\Omega}^F(\bx):=\emptyset$ if $\bx\notin \Omega$ and $\partial^F{f}(\bx):=\emptyset$ if $\bx\notin\dom f$.
For a norm $\|\cdot\|$ on a vector space $X$,
it is well known \cite[Example~3.36]{MorNam22} that
	\begin{equation}\label{P4.12}
		\partial\|\cdot\|(x) = 
		\begin{cases}
			\left\{x^*\in X^*\mid \|x^*\| =1\;\;\text{and}\;\;\langle x^*,x\rangle=\|x\|\right\} & \text{if } x \ne 0_{X},\\
			\overline\B_{X^*}  & \text{otherwise }.
		\end{cases}
	\end{equation}
A function $f$ is said to be \textit{strictly convex} if $f\left(\frac{x+x'}{2}\right)<\frac{f(x)+f(x')}{2}$
for any distinct points $x,x'\in\dom f$.
A norm $\|\cdot\|$ on a vector space $X$ is
\textit{strictly convex }if $\norm{{x+x'}}<2$ for any distinct points
$x,x'\in\mathbb{S}_X$.

\section{Primal characterizations of sign-symmetric norms}\label{S3}
In this section, we establish a correspondence between a class of sign-symmetric (strictly convex) norms and a class of (strictly) convex continuous functions.

Throughout this paper, we also assume  that a sign-symmetric norm $\vertiii{\cdot}$ on $X^n$ and the given norm $\|\cdot\|$ on $X$
satisfy the following compatibility condition:
	\begin{gather}\label{A2}\tag{{A2}}
		\vertiii{(0_X,\ldots,0_X,v,0_X,\ldots,0_X)}=\|v\|\;\;\text{for all}\;\;v\in X,
	\end{gather}	
where $v$ is in the $i$th position for $i=1,\ldots,n$.
We denote by ${\rm \textbf{N}}_{X^n}$ the family of all norms on $X^n$ that satisfy conditions \eqref{A1} and \eqref{A2}, and by ${\rm \textbf{N}}^{\rm{sc}}_{X^n}$ the subclass of ${\rm \textbf{N}}_{X^n}$ consisting of all strictly convex norms satisfying  \eqref{A1} and \eqref{A2}.

\begin{remark}
	\begin{enumerate}[\rm (i)]
		\item	
		For a given $p\in[1,\infty]$,  the $p$-norm on 
		$X^n$, defined by  
			\begin{equation}\label{pnorm}
				\vertiii{x}_{p}
				:= \begin{cases}
					\left(\|x_1\|^p+\ldots\|x_n\|^p\right)^{\frac{1}{p}}  & \text{if } p\in[1,\infty),\\	
					\max\{\|x_1\|,\ldots,\|x_n\|\} & \text{if } p=\infty,
				\end{cases} 
			\end{equation}
		for all $x:=(x_1,\ldots,x_n)\in X^n$, belongs to ${\rm \textbf{N}}_{X^n}$.
		\item 
		Recall from \cite{SaiKatTak00} that a norm $\vertiii{\cdot}$ on the complex space $\mathbb{C}^n$ is called \textit{absolute} if $\vertiii{(x_1,\ldots,x_n)}=\vertiii{(\abs{x_1},\ldots,\abs{x_n})}$
		for all $(x_1,\ldots,x_n)\in\mathbb{C}^n$,
		where $|\cdot|$ is the standard norm on $\mathbb{C}$.
		It is obvious that every absolute norm on $\mathbb{C}^n$  is sign-symmetric. 
		The converse implication does not generally hold (Example~\ref{E3.2}).
		A norm $\vertiii{\cdot}$ on $\mathbb{C}^n$ is called \textit{normalized} \cite{SaiKatTak00}  if 
		$\vertiii{\mathbf{e}_1}=\ldots=\vertiii{\mathbf{e}_n}=1$, where $\mathbf{e}_i:=(0,\ldots,0,1,0,\ldots,0)$ with the value $1$ in the $i$th position  $(i=1,\ldots,n)$. 
		Thus, condition \eqref{A2} can be seen as a vector counterpart of this normalization condition.
	\end{enumerate}
\end{remark}	

The next example shows that a sign-symmetric norm on $\mathbb{C}^2$ is not necessarily absolute.
\begin{example}\label{E3.2}
	Consider the norm $\vertiii{(x_1,x_2)}:=\frac{\abs{x_1-x_2}+\abs{x_1+x_2}}{2}	$ for all $(x_1,x_2)\in\mathbb{C}^2.$
	It is clear that $\vertiii{\cdot}\in {\rm \textbf{N}}_{\mathbb{C}^2}$.
	However, this norm is not absolute as $\vertiii{(x_1,x_2)}=2\ne\sqrt{2}=\vertiii{(\abs{x_1},\abs{x_2})}$, where $x_1:=1+i$ and $x_2:=1-i$.
\end{example}	

The following proposition establishes quantitative relations between norms in $\textbf{{\rm \textbf{N}}}_{X^n}$.
\begin{proposition}\label{L3.2}
	Let $\vertiii{\cdot}\in \textbf{{\rm \textbf{N}}}_{X^n}$ and $x:=(x_1,\ldots,x_n)\in X^n$.
	Then
	\begin{enumerate}[\rm (i)]
		\item\label{L3.2-1}
		$\max\{\vertiii{(0_X,x_2,\ldots,x_n)},\ldots,\vertiii{(x_1,\ldots,x_{n-1},0_X)}\}\le  \vertiii{x}$;
		\item\label{L3.2-0}
		$\vertiii{(x_1,0_X,\ldots,0_X)}\le \vertiii{(x_1,x_2,0_X\ldots,0_X)}\le\ldots\le\vertiii{(x_1,\ldots,x_{n-1},0_X)}\le\vertiii{x}$;
		\item\label{L3.2-2}
		$\vertiii{x}_{\infty}\le \vertiii{x}\le\vertiii{x}_1  \le n\cdot \vertiii{x}_\infty$, where $\vertiii{\cdot}_1$ and $\vertiii{\cdot}_\infty$ are given by \eqref{pnorm}. 
	\end{enumerate}
\end{proposition}	
\begin{proof}
	%Assume that $\vertiii{\cdot}\in \textbf{{\rm \textbf{N}}}_{X^n}$ and  $x:=(x_1,\ldots,x_n)\in X^n$.
	By \eqref{A1},
	$\vertiii{(0_X,x_2,\ldots,x_n)}\le \frac{1}{2}\left(\vertiii{x}+\vertiii{(-x_1,x_2,\ldots,x_n)}\right)=\vertiii{x}.$
	The other inequalities in \eqref{L3.2-1} and  \eqref{L3.2-0}  are proved similarly.
	By \eqref{A1} and \eqref{A2}, 
		\begin{align*}
			\vertiii{x}_{\infty}
			=\max_{1\le i\le n}\|x_i\|
			\le\vertiii{x}
			\le\sum_{i=1}^{n}\vertiii{(0_X,\ldots,0_X,x_i,0_X,\ldots,0_X)}
			=\vertiii{\cdot}_1\le n\cdot\vertiii{\cdot}_\infty.
		\end{align*}
This proves \eqref{L3.2-2}.
The proof is complete.
\end{proof}	
We are going to study relations between the class of sign-symmetric norms and a class of convex continuous functions.
Consider a convex and compact   subset of $\R^n$ given by
	\begin{gather*}
		\Omega_n:=\{(t_1,\ldots,t_{n})\in\R^{n}\mid t_1,\ldots,t_n\ge 0
		\;\;\text{and}\;\; t_1+\ldots+t_{n}=1\}.
	\end{gather*}
Let $\pmb{\Psi}_n$ denote the class of all convex continuous  functions $\psi:\Omega_n\to\R$ satisfying the following conditions:
	\begin{gather}\label{S3.2-1}\tag{B1}
		\psi(\mathbf{e}_1)=\ldots=\psi(\mathbf{e}_n)=1,\\\label{S3.2-2} \tag{B2}
		\psi(t)\ge (1-t_i)\cdot\psi\left(\dfrac{t_1}{1-t_i},\ldots,\dfrac{t_{i-1}}{1-t_i},0,\dfrac{t_{i+1}}{1-t_i},\ldots,\dfrac{t_{n}}{1-t_i}\right)
		\;\;(i=1,\ldots,n)
	\end{gather}	
for all $t:=(t_1,\ldots,t_n)\in\Omega^\circ_n:=\{(t_1,\ldots,t_n)\in\Omega_n\mid t_1,\ldots,t_n<1\}$.
Denote by $\pmb{\Psi}^{\rm{sc}}_n$ the subclass of $\pmb{\Psi}_n$ consisting of all strictly convex functions satisfying \eqref{S3.2-1} and \eqref{S3.2-2}.

\begin{remark}\label{R3}
	\begin{enumerate}[\rm (i)]
		\item\label{R3.1}
		Consider a subset of $\R^{n}$ given by
			\begin{gather*}
				\Omega'_{n}:=\{(t_1,\ldots,t_{n})\in\R^{n}\mid t_1,\ldots,t_{n}\ge 0\;\;\text{and}\;\; t_1+\ldots+t_{n}\le1\}.
			\end{gather*}
		In \cite{SaiKatTak00}, the authors study  a class $\pmb{\Psi}'_{n-1}$ consisting of all convex continuous functions $\varphi:\Omega'_{n-1}\to \R$ satisfying the following conditions:
			\begin{gather}\label{S3.2-4}\tag{B1$^\prime$}
				\varphi(0,0,\ldots,0)=\varphi(1,0,\ldots,0)=\ldots=\varphi(0,\ldots,0,1)=1,\\\label{S3.2-5}\tag{B2$^\prime$}
				\varphi(t)\ge (1-t_i)\cdot\varphi\left(\dfrac{t_1}{1-t_i},\ldots,\dfrac{t_{i-1}}{1-t_i},0,\dfrac{t_{i+1}}{1-t_i},\ldots,\dfrac{t_{n-1}}{1-t_i}\right)\;(i=1,\ldots,n-1),\\\label{S3.2-6}\tag{B3$^\prime$}
				\text{and}\;\;\varphi(t)\ge \left(1-t_n\right)\cdot\varphi\left(\dfrac{t_1}{1-t_n},\ldots,\ldots,\dfrac{t_{n-1}}{1-t_n}\right)
			\end{gather}	
		for all $t:=(t_1,\ldots,t_{n-1})\in \Omega'_{n-1}$ with $t_1,\ldots,t_{n-1}<1$ in \eqref{S3.2-5} and $t_n:=\sum_{i=1}^{n-1}t_i<1$ in \eqref{S3.2-6}.
		The class  $\pmb{\Psi}'_{n-1}$ is employed to characterize {normalized} and {absolute} norms on $\mathbb{C}^n$; see, e.g., \cite[Theorem~3.4]{SaiKatTak00}.
		The special case  $\pmb{\Psi}'_1$ was originally introduced in  \cite{BonDun73} for the study of numerical ranges in $\mathbb{C}^2$.
		The quasi-symmetric conditions specified in \eqref{S3.2-5} and \eqref{S3.2-6} result in the corresponding quasi-symmetric representation of the norm on $\mathbb{C}^n$.
		In this paper, using the class $\pmb{\Psi}_n$, we establish a full symmetric representation for norms on general vector spaces.
		\item\label{R3.2}
		The families $\pmb{\Psi}_n$ and $\pmb{\Psi}'_{n-1}$ are in a one-to-one correspondence.
		Indeed, given a $\psi\in \pmb{\Psi}_n$, the function
		$\varphi'(t_1,\ldots,t_{n-1}):=\psi(t_1,\ldots,t_{n-1},1-\sum_{i=1}^{n-1}t_i)$ for all $(t_1,\ldots,t_{n-1})\in\Omega'_{n-1}$ belongs to $\pmb{\Psi}'_{n-1}$.
		On the other hand, given a $\varphi\in \pmb{\Psi}'_{n-1}$, the function
		$\psi(t_1,\ldots,t_n):=\varphi(t_1,\ldots,t_{n-1})$ for all $(t_1,\ldots,t_n)\in\Omega_n$ belongs to $\pmb{\Psi}_n$.
	\end{enumerate}
\end{remark}	

The next result is of importance.
It is a modified version of \cite[Lemma~3.2]{SaiKatTak00}.
\begin{proposition}\label{P5.1}
	Let $\psi\in\pmb{\Psi}_n$.
	Then $\frac{1}{n}\le \max\{t_1,\ldots,t_n\}\le\psi(t)\le 1$
	for all $t:=(t_1,\ldots,t_n)\in\Omega_n$.
\end{proposition}	
\vspace{-0.3\baselineskip}
\begin{proof}
	Let $\psi\in\pmb{\Psi}_n$ and $t:=(t_1,\ldots,t_n)\in\Omega_n$.
	The first inequality is obvious.
	The convexity of $\psi$ implies that
	$\psi(t)\le\sum_{i=1}^{n} t_i\psi(\mathbf{e}_i)=\sum_{i=1}^{n}t_i=1.
	$
	Let $t_{i_0}:=\max\{t_1,\ldots,t_n\}>0$ for some $i_0\in\{1,\ldots,n\}$.
	By \eqref{S3.2-2}, 
	\begin{align*}
		\psi(t)
		&\ge(1-t_1)\cdot\psi\left(0,\dfrac{t_2}{1-t_1},\ldots,\dfrac{t_{i_0}}{1-t_1},\ldots,\dfrac{t_{n}}{1-t_1}\right)\\
		&\ge (1-t_1-t_2)\cdot\psi\left(0,0,\dfrac{t_3}{1-t_1-t_2},\ldots,\dfrac{t_{i_0}}{1-t_1-t_2},\ldots,\dfrac{t_{n}}{1-t_1-t_2}\right)\\
		&\;\;\vdots\\
		&\ge \left(1-\sum_{j=1}^{i_0-1}t_j\right)\cdot\psi\left(0,\ldots,0,\dfrac{t_{i_0}}{1-\sum_{j=1}^{i_0-1}t_j},\ldots,\dfrac{t_{n}}{1-\sum_{j=1}^{i_0-1}t_j}\right)\\
		&\ge \left(1-\sum_{i_0\ne j=1}^{i_0+1}t_j\right)\cdot\psi\left(0,\ldots,0,\dfrac{t_{i_0}}{1-\sum_{i_0\ne j=1}^{i_0+1}t_j},0,\dfrac{t_{i_0+2}}{\sum_{i_0\ne j=1}^{i_0+1}t_j}\ldots,\dfrac{t_{n}}{1-\sum_{i_0\ne j=1}^{i_0+1}t_j}\right)\\
		&\;\;\vdots\\
		&\ge \left(1-\sum_{i_0\ne j=1}^{n}t_j\right)\cdot\psi\left(0,\ldots,0,\dfrac{t_{i_0}}{1-\sum_{i_0\ne j=1}^{n}t_j},0,\ldots,0\right)
		=t_{i_0}\cdot\psi(\mathbf{e}_{i_0}).
	\end{align*}	
	In view of \eqref{S3.2-1}, we have $\psi(t)\ge t_{i_0}$.
	This completes the proof.
\end{proof}	

\begin{proposition}\label{L5.2}
	Let $\psi\in\pmb{\Psi}_n$.
	The following assertions hold:
	\begin{enumerate}[\rm (i)] 
		\item\label{L5.2.2}
		for any $t:=(t_1,\ldots,t_{n})$ and $t':=(t'_1,\ldots,t'_{n})$ in $\Omega_n$ satisfying
			\begin{gather}\label{L5.2-11}
				\dfrac{t'_1}{t_1}=\ldots=\dfrac{t'_{i-1}}{t_{i-1}}=\dfrac{1-t'_i}{1-t_i}=\dfrac{t'_{i+1}}{t_{i+1}}=\ldots=\dfrac{t'_{n}}{t_{n}}=\gamma
			\end{gather}	
		for some $\gamma\ge1$, one has
			\begin{gather}\label{L5.2-1}
				\dfrac{\psi(t')}{1-t'_i}\le\dfrac{\psi(t)}{1-t_i}\;\;
				\text{and}\;\;\dfrac{\psi(t')}{t'_j}\le\dfrac{\psi(t)}{t_j}\;\text{for all}\; j\ne i.
			\end{gather}	
		\item\label{L5.2.3}
		if $\psi\in \pmb{\Psi}^{\rm{sc}}_n$ and condition \eqref{L5.2-11} holds for some $\gamma>1$, then inequalities \eqref{L5.2-1} are strict.
	\end{enumerate}
\end{proposition}	

\begin{proof}
	Let $\psi\in\pmb{\Psi}_n$ and $t:=(t_1,\ldots,t_{n}),t':=(t'_1,\ldots,t_{n})\in\Omega_n$ satisfy condition \eqref{L5.2-11} for some $\gamma\ge1$.
	Then
	$t'=\frac{(1-t_i)(\ga-1)}{t_i}\cdot\frac{t-\nu}{1-t_i}+\frac{1-\ga(1-t_i)}{t_i}\cdot t$, where
	$\nu:=(0,\ldots,0,t_i,0,\ldots,0)$ with $t_i$ being in the $i$th position.
	Observe that $\frac{(1-t_i)(\ga-1)}{t_i}+\frac{1-\ga(1-t_i)}{t_i}=1$.
	By the convexity of $\psi$, we have
	\begin{gather}\label{L5.2-14}
		\psi(t')\le\dfrac{(1-t_i)(\ga-1)}{t_i}\cdot\psi\left(\dfrac{t-\nu}{1-t_i}\right)+\dfrac{1-\ga(1-t_i)}{t_i}\cdot \psi(t).	
	\end{gather}	
	By \eqref{L5.2-11} and \eqref{L5.2-14},
	\begin{align*}
		\dfrac{\psi(t)}{1-t_i}-	\dfrac{\psi(t')}{1-t'_i}
		= \dfrac{\psi(t)}{1-t_i}-\dfrac{\psi(t')}{\ga (1-t_i)}
		\ge \dfrac{\gamma-1}{\ga t_i(1-t_i)}\cdot \left[\psi(t)-(1-t_i)\psi\left(\dfrac{t-\nu}{1-t_i}\right)\right]\ge0,
	\end{align*}	
	where the last inequality follows from \eqref{S3.2-2}.
	Thus,  inequalities \eqref{L5.2-1} hold.
	If $\psi\in \pmb{\Psi}^{\rm{sc}}_n$ and condition \eqref{L5.2-11} holds for some $\gamma>1$,
	then inequality \eqref{L5.2-14} is strict.
	Consequently,  inequalities \eqref{L5.2-1} are strict.
\end{proof}

\begin{remark}
	It is claimed in \cite[p. 37]{BonDun73} that 
	the functions $\frac{\varphi(t)}{t}$ and $\frac{\varphi(t)}{1-t}$ are decreasing and increasing on $(0,1]$ and $[0,1)$, respectively, for any $\varphi\in \pmb{\Psi}'_1$.
	These facts are direct consequences of Proposition~\ref{L5.2}\;\eqref{L5.2.2}.
\end{remark}	

\begin{lemma}\label{C5.3}
	Let $0<\al_i\le\be_i$ $(i=1,\ldots,n)$, $\al:=\sum_{i=1}^{n}\al_i$ and 
	$\be:=\sum_{i=1}^{n}\be_i$.
	\begin{enumerate}[\rm (i)]
		\item\label{C5.3.1}
		If $\psi\in\pmb{\Psi}_n$, then
			\begin{gather}\label{L5.2-12}
				\al \cdot\psi\left(\dfrac{\al_1}{\al},\ldots,\dfrac{\al_{n}}{\al}\right)\le \be \cdot\psi\left(\dfrac{\be_1}{\be},\ldots,\dfrac{\be_{n}}{\be}\right).
			\end{gather}		
		\item\label{C5.3.2}
		If $\psi\in \pmb{\Psi}^{\rm{sc}}_n$ and $\al_i<\be_i$ for some $i\in\{1,\ldots,n\}$, then 
		inequality \eqref{L5.2-12} is strict.
	\end{enumerate}
\end{lemma}	

\begin{proof}
	Suppose that $\psi\in\pmb{\Psi}_n$.
	Let
	$\widehat\al_j:=\sum_{i=j}^{n}\al_i$ and $\widehat\be_j:=\sum_{i=1}^{j}\be_i$ for $j=1,\ldots,n$, and  $\xi_{j}:=\widehat\be_j+\widehat\al_{j+1}$ for $j=1,\ldots,n-1$. 
	Then $\widehat{\al}_1\le \xi_1\le \ldots\le\xi_{n-1}\le\widehat{\be}_n$.
	Inequality \eqref{L5.2-12} is equivalent to
	$\widehat\al_1\cdot\psi\left(\frac{\al_1}{\widehat\al_1},\ldots,\frac{\al_{n}}{\widehat\al_1}\right)\le \widehat\be_{n} \cdot\psi\left(\frac{\be_1}{\widehat\be_n},\ldots,\frac{\be_{n}}{\widehat\be_n}\right).
	$
	We repeatedly apply Proposition~\ref{L5.2}\;\eqref{L5.2.2} in the remainder of the proof.
	For
	$t':=\left(\frac{\al_1}{\widehat\al_1},\ldots,\frac{\al_{n}}{\widehat\al_1}\right)$ and $t:=\left(\frac{\be_1}{ \xi_1},\frac{\al_2}{ \xi_1},\ldots,\frac{\al_{n}}{ \xi_1}\right)$, we have
	\begin{gather}\label{L5.2-6}
		\widehat\al_1\cdot\psi\left(\frac{\al_1}{\widehat\al_1},\ldots,\frac{\al_{n}}{\widehat\al_1}\right)\le\xi_1 \cdot\psi\left(\frac{\be_1}{ \xi_1},\frac{\al_2}{ \xi_1},\ldots,\frac{\al_{n}}{ \xi_1}\right).
	\end{gather}	
	For $t':=\left(\frac{\be_1}{ \xi_1},\frac{\al_2}{ \xi_1},\ldots,\frac{\al_{n}}{ \xi_1}\right)$ and $t:=\left(\frac{\be_1}{\xi_2},\frac{\be_2}{\xi_2},\frac{\al_3}{\xi_2},\ldots,\frac{\al_n}{\xi_2}\right)$, we have
	\begin{gather*}
		\xi_1 \cdot\psi\left(\frac{\be_1}{ \xi_1},\frac{\al_2}{ \xi_1},\ldots,\frac{\al_{n}}{ \xi_1}\right)\le\xi_2 \cdot\psi\left(\dfrac{\be_1}{\xi_2},\dfrac{\be_2}{\xi_2},\dfrac{\al_3}{\xi_2},\ldots,\dfrac{\al_{n}}{\xi_2}\right).
	\end{gather*}	
	Repeating this process, we obtain
	\begin{gather*}
		\xi_{n-2} \cdot\psi\left(\dfrac{\be_1}{\xi_{n-2}},\ldots,\dfrac{\be_{n-2}}{ \xi_{n-2} },\dfrac{\al_{n-1}}{\xi_{n-2}},\dfrac{\al_n}{\xi_{n-2}}\right)\le\xi_{n-1} \cdot\psi\left(\dfrac{\be_1}{\xi_{n-1}},\ldots,\dfrac{\be_{n-1}}{ \xi_{n-1}},\dfrac{\al_n}{\xi_{n-1}}\right).
	\end{gather*}	
	Finally, for $t':=\left(\frac{\be_1}{\xi_{n-1}},\ldots,\frac{\be_{n-1}}{ \xi_{n-1}},\frac{\al_n}{\xi_{n-1}}\right)$ and $t:=\left(\frac{\be_1}{\widehat\be_n},\ldots,\frac{\be_{n}}{\widehat\be_n}\right)$, we have
	\begin{gather*}
		\xi_{n-1} \cdot\psi\left(\dfrac{\be_1}{\xi_{n-1}},\ldots,\dfrac{\be_{n-1}}{ \xi_{n-1}},\dfrac{\al_n}{\xi_{n-1}}\right)\le \widehat\be_{n} \cdot\psi\left(\dfrac{\be_1}{\widehat\be_n},\ldots,\dfrac{\be_{n}}{\widehat\be_n}\right).
	\end{gather*}	
	Thus, assertion \eqref{C5.3.1} is proved.
	If $\psi\in \pmb{\Psi}^{\rm{sc}}_n$, then without loss of generality, we can assume that $\al_1<\be_1$.
	Thus, $\ga_1>1$.
	By Proposition~\ref{L5.2}\;\eqref{L5.2.3}, inequality \eqref{L5.2-6} is strict, and consequently, inequality \eqref{L5.2-12} is strict.
\end{proof}	

The following theorems establish correspondences between ${\rm \textbf{N}}_{X^n}$
and $\pmb{\Psi}_n$, $\textbf{{\rm \textbf{N}}}^{\rm{sc}}_{X^n}$ and  $\pmb{\Psi}^{\rm{sc}}_n$.

\begin{theorem}\label{T2.7}
	Let $\vertiii{\cdot}$ be a norm on $X^n$, and $\mathbf{u}:=(\mathbf{u}_1,\ldots,\mathbf{u}_n)\in \mathbb{S}^n_X$.
	Define 
	\begin{gather}\label{T3.10.0}
		\psi_{\mathbf{u}}(t):=\vertiii{(t_1\mathbf{u}_1,\ldots,t_{n}\mathbf{u}_n)}
		\;\;\text{for all}\;\;t:=(t_1,\ldots,t_{n})\in\Omega_n.
	\end{gather}
	\begin{enumerate}[\rm (i)] 
		\item\label{T2.6-1}
		If $\vertiii{\cdot}\in\textbf{{\rm \textbf{N}}}_{X^n}$, then
		$\psi_{\mathbf{u}}\in\pmb{\Psi}_n$.
		\item\label{T2.6-2}
		If $\vertiii{\cdot}\in\textbf{{\rm \textbf{N}}}^{\rm{sc}}_{X^n}$, then $ \psi_{\mathbf{u}}\in\pmb{\Psi}^{\rm{sc}}_{n}$.
	\end{enumerate}
\end{theorem}	
\begin{proof}
	Assume that $\vertiii{\cdot}\in\textbf{{\rm \textbf{N}}}_{X^n}$, $\mathbf{u}:=(\mathbf{u}_1,\ldots,\mathbf{u}_n)\in \mathbb{S}^n_X$, and $\psi_{\mathbf{u}}$ is given by \eqref{T3.10.0}.
	The continuity and  convexity of $\psi_{\mathbf{u}}$ follow from the fact that it is the composition of the norm $\vertiii{\cdot}$ and the affine mapping
	$t\mapsto (t_1\mathbf{u}_1,\ldots,t_{n}\mathbf{u}_n)$.
	For each $i=1,\ldots,n$, by Proposition~\ref{L3.2}\;\eqref{L3.2-1},
	we have
	\begin{align*}
		\psi_{\mathbf{u}}(s)
		&\ge \vertiii{(s_1\mathbf{u}_1,\ldots,s_{i-1}\mathbf{u}_{i-1},0_X,s_{i+1}\mathbf{u}_{i+1},\ldots, s_{n}\mathbf{u}_n)}\\
		&=(1-s_i)\cdot\vertiii{\left(\frac{s_1}{1-s_i}\mathbf{u}_1,\ldots,\frac{s_{i-1}}{1-s_i}\mathbf{u}_{i-1},0_X,\frac{s_{i+1}}{1-s_i}\mathbf{u}_{i+1},\ldots, \frac{s_n}{1-s_i}\mathbf{u}_n\right)}\\
		&=(1-s_i)\cdot\psi_{\mathbf{u}}\left(\frac{s_1}{1-s_i},\ldots,\frac{s_{i-1}}{1-s_i},0,\frac{s_{i+1}}{1-s_i},\ldots,\frac{s_{n}}{1-s_i}\right)\; (i=1,\ldots,n)
	\end{align*}	
	for any $s:=(s_1,\ldots,s_{n})\in\Omega^\circ_n.$
	Besides,
	$\psi_{\mathbf{u}}(\mathbf{e}_i)=\vertiii{(0_X,\ldots,0_X,\mathbf{u}_i,0_X,\ldots,0_X)}=\|\mathbf{u}_i\|=1$
	$(i=1,\ldots,n)$.
	Thus, assertion  \eqref{T2.6-1} is proved.
	
	Let $\vertiii{\cdot}\in\textbf{{\rm \textbf{N}}}^{\rm{sc}}_{X^n}$.
	To show that $\psi_{\mathbf{u}}\in\pmb{\Psi}^{\rm{sc}}_n$,
	it is sufficient to prove that $\psi_{\mathbf{u}}$ is strictly convex.	
	Suppose that 
	$\psi_{\mathbf{u}}\left(\frac{s+s'}{2}\right)=\frac{\psi_{\mathbf{u}}(s)+\psi_{\mathbf{u}}(s')}{2}$
	for some distinct points $s:=(s_1,\ldots,s_{n}), s':=(s'_1,\ldots,s'_{n})\in \Omega_n$.
	Let $x:=\left(\frac{s_1}{2}\textbf{u}_1,\ldots,\frac{s_n}{2}\textbf{u}_n\right)$ and $x':=\left(\frac{s'_1}{2}\textbf{u}_1,\ldots,\frac{s'_n}{2}\textbf{u}_n\right)$.
	Then $\vertiii{x+x'}=\vertiii{x}+\vertiii{x'}$.
	By the strict convexity of $\vertiii{\cdot}$, there exists  a scalar $\la> 0$ such that $x'=\la x$.
	Thus, $s_i=\lambda s'_i$
	$(i=1,\ldots,n)$, and consequently,
	$1=s_1+\ldots+s_n=\la(s'_1+\ldots+s'_n)=\la$.
	This contradicts the assumption that $s$ and $s'$ are distinct.
	Thus, assertion \eqref{T2.6-2} is proved.
\end{proof}	

\begin{theorem}\label{T3.10}
	Let $\psi\in\pmb{\Psi}_n$.
	Define the function $\vertiii{\cdot}_\psi:X^n\to\R$ by
	\begin{equation}\label{T3.10-1}
		\vertiii{x}_\psi
		:= \begin{cases}
			\left(\sum_{i=1}^n\|x_i\|\right)\cdot\psi\left(\dfrac{\|x_1\|}{\sum_{i=1}^n\|x_i\|},\ldots,\dfrac{\|x_{n}\|}{\sum_{i=1}^n\|x_i\|}\right)  & \text{if } x\ne0_{X^n},\\
			0 & \text{otherwise}
		\end{cases} 
	\end{equation}
	for all $x:=(x_1,\ldots,x_n)\in X^n$.
	The following assertions hold:
	\begin{enumerate}[\rm (i)]
		\item \label{T3.10.2}
		$\vertiii{\cdot}_\psi \in\textbf{{\rm \textbf{N}}}_{X^n}$;
		\item\label{T3.10.1}
		if $\psi\in \pmb{\Psi}^{\rm{sc}}_{X^n}$ and  $\|\cdot\|$ is strictly convex, then $\vertiii{\cdot}_\psi\in \textbf{{\rm \textbf{N}}}^{\rm{sc}}_{X^n}$;
		\item\label{T3.10.3}
		$\psi(s)=\vertiii{(s_1\mathbf{u}_1,\ldots,s_{n}\mathbf{u}_n)}_{\psi}$
		for all $(\mathbf{u}_1,\ldots,\mathbf{u}_n)\in \mathbb{S}^n_X$ and $s:=(s_1,\ldots,s_{n})\in\Omega_n$.
	\end{enumerate}	
\end{theorem}	

\begin{proof}
	\begin{enumerate}[\rm (i)]
		\item		 
		It is clear that $\vertiii{\cdot}_\psi$ satisfies condition \eqref{A1}.
		By  \eqref{T3.10-1}, we have
		$
		\vertiii{(0_X,\ldots,0_X,v,0_X,\ldots,0_X)}_\psi =\|{v}\|\cdot\psi(\mathbf{e}_i)=\|v\|
		$ $(i=1,\ldots,n)$
		for all $v\in X$.
		Thus, $\vertiii{\cdot}_\psi$ satisfies condition \eqref{A2}.
		To show that $\vertiii{\cdot}_\psi$ is a norm, it suffices to verify the triangle inequality as the other norm conditions hold trivially.
		Let $x:=(x_1,\ldots,x_n),x':=(x'_1,\ldots,x'_n)\in X^n$.
		We  prove that
		$
		\vertiii{ x+x'}_\psi \le \vertiii{ x}_\psi +\vertiii{x'}_\psi.
		$
		This condition is obviously satisfied if any of the following conditions holds:
		1) $x=0_{X^n}$; 2) $x'=0_{X^n}$; 3) $x+x'=0_{X^n}$.
		Otherwise, let 
		\begin{gather}\label{T3.10-3}
			\al:=\sum_{i=1}^{n}\|x_i\|,\; \be:=\sum_{i=1}^{n}\|x'_i\|,\; t:=\left(\frac{\|x_1\|}{\al},\ldots,\frac{\|x_{n}\|}{\al}\right),\; t':=\left(\frac{\|x'_1\|}{\be},\ldots,\frac{\|x'_{n}\|}{\be}\right).
		\end{gather}
		By Lemma~\ref{C5.3}\;\eqref{C5.3.1} and the convexity of $ \psi$, 
		\begin{align}\label{T5.7-2}
			\vertiii{ x+x'}_\psi
			&=\left(\sum_{i=1}^{n}\|x_i+x'_i\|\right)\cdot\psi\left(\dfrac{
				\|x_1+x'_1\|}{\sum_{i=1}^{n}\|x_i+x'_i\|},\ldots,\dfrac{
				\|x_{n}+x'_{n}\|}{\sum_{i=1}^{n}\|x_i+x'_i\|}\right)\\ \notag
			&\le(\al+\be)\psi\left(\dfrac{
				\|x_1\|+\|x'_1\|}{\al+\be},\ldots,\dfrac{
				\|x_{n}\|+\|x'_{n}\|}{\al+\be}\right)\\ \notag
			&=(\al+\be)\psi\left[\dfrac{\al}{\al+\be}\cdot t+\dfrac{\be}{\al+\be}\cdot t'\right]
			\le \al\psi(t)+\be\psi(t')=\vertiii{ x}_\psi +\vertiii{x'}_\psi.
		\end{align}	
		Thus, $\vertiii{\cdot}_\psi\in \textbf{{\rm \textbf{N}}}_{X^n}$.
		\item
		Let $\psi\in \pmb{\Psi}^{\rm{sc}}_n$ and $\|\cdot\|$ be strictly convex.
		Take distinct points $x:=(x_1,\ldots,x_n), x':=(x'_1,\ldots,x'_n)\in X^n$ satisfying $\vertiii{x}_\psi=\vertiii{x'}_\psi=1$ .
		We prove that
		$\vertiii{x+x'}_\psi<2.$
		Let $\al$, $\be$, $t$ and $t'$ be given by \eqref{T3.10-3}.
		Note that $t,t'\in\Omega_n$.
		We consider two cases.\\
		\textbf{Case 1}: $t\ne t'$.
		By Lemma~\ref{C5.3}\;\eqref{C5.3.1} and the 
		strict convexity of $\psi$, we obtain
		\begin{align*}
			\vertiii{x+x'}_\psi\le(\al+\be)\psi\left(\dfrac{\al}{\al+\be}\cdot t+\dfrac{\be}{\al+\be}\cdot t' \right)
			< \al\psi(t)+\be\psi(t')=2.
		\end{align*}	
		\textbf{Case 2}: $t=t'$.
		Then $\al=\be$, and consequently, $\|x_i\|=\|x'_i\|\;\;(i=1,\ldots,n).$
		Suppose that $\vertiii{x+x'}_\psi=2$.
		We have
		\begin{align*}
			2=\vertiii{x+x'}_\psi\le(\al+\be)\psi\left(\dfrac{
				\|x_1\|+\|x'_1\|}{\al+\be},\ldots,\dfrac{
				\|x_{n}\|+\|x'_{n}\|}{\al+\be}\right)\le2.
		\end{align*}	
		Thus,
		$
		\vertiii{x+x'}_\psi=(\al+\be)\psi\left(\frac{
			\|x_1\|+\|x'_1\|}{\al+\be},\ldots,\frac{	\|x_{n}\|+\|x'_{n}\|}{\al+\be}\right).
		$
		By Lemma~\ref{C5.3}\;\eqref{C5.3.2},  $\|x_i+x'_i\|=\|x_i\|+\|x'_i\|$ $(i=1,\ldots,n)$.
		The strict convexity of $\|\cdot\|$ ensures the existence of scalars $\la_i> 0$ such that $x'_i=\la_i x_i$ $(i=1,\ldots,n)$.
		Therefore, $\la_1=\ldots=\la_n=1$ meaning $x=x'$, which
		contradicts the assumption that $x$ and $x'$ are distinct.
		\item 
		Let $(\mathbf{u}_1,\ldots,\mathbf{u}_n)\in \mathbb{S}^n_X$ and $s:=(s_1,\ldots,s_{n})\in\Omega_n$.
		Substituting
		$(x_1,\ldots,x_n):=(s_1\mathbf{u}_1,\ldots,s_n\mathbf{u}_n)$ into \eqref{T3.10-1}, we obtain the desired equality.
	\end{enumerate}
	The proof is complete.
\end{proof}	

\begin{remark}\label{E3.9}
	\begin{enumerate}[\rm (i)]
		\item	
		The $p$-norm defined by \eqref{pnorm} is obtained by substituting the following function  into \eqref{T3.10-1}:
		\begin{equation}\label{pf}
			\psi_{p}(t)
			:= \begin{cases}
				\left(t_1^p+\ldots+t_n^p\right)^{\frac{1}{p}}  & \text{if } p\in[1,\infty),\\
				\max\{t_1,\ldots,t_n\} & \text{if } p=\infty
			\end{cases} 
		\end{equation}
		for all $t:=(t_1,\ldots,t_n)$.
		\item
		A conventional approach to constructing a norm on $X^n$ is to combine a norm $\vertiii{\cdot}_{\R^n}$ on $\mathbb{R}^n$
		with the given norm $\|\cdot\|$ on $X$.
		Consider the function $	\vertiii{x}:=\vertiii{(\|x_1\|,\ldots,\|x_n\|)}_{\R^n}$
		for all $x:=(x_1,\ldots,x_n)\in X^n$.
		In general, this function may not be a norm on  $X^n$ since it may fail to satisfy the triangle inequality; see \cite{CuoKru24-2} for a counterexample.
		However, if the norm $\vertiii{\cdot}_{\R^n}$ is monotone  in the sense that  $\vertiii{(\al_1,\ldots,\al_n)}_{\R^n}\le\vertiii{(\be_1,\ldots,\be_n)}_{\R^n}$ for $\abs{\al_i}\le\abs{\be_i}$ $(i=1,\ldots,n)$,
		then $\vertiii{\cdot}$ is indeed a norm on $X^n$;
		see, e.g., \cite[p.134]{LooSte90}.
	\end{enumerate}
\end{remark}	

The next theorem shows that if a norm
on $X^n$ is of the form \eqref{T3.10-1} for some function  $\psi$ on $\Omega_n$, then
$\psi$ must belong to $\pmb{\Psi}_n$ ($\pmb{\Psi}^{\rm{sc}}_n$) whenever the norm belongs to
$\textbf{{\rm \textbf{N}}}_{X^n}$ ($\textbf{{\rm \textbf{N}}}^{\rm{sc}}_{X^n}$).

\begin{theorem}\label{P3.10}
	Let $\psi:\Omega_n\to\R$, and the function $\vertiii{\cdot}_\psi$ be defined by \eqref{T3.10-1}.
	\begin{enumerate}[\rm (i)]
		\item\label{P3.20-1}
		If $\vertiii{\cdot}_\psi\in\textbf{{\rm \textbf{N}}}_{X^n}$, then $ \psi\in\pmb{\Psi}_n$.
		\item\label{P3.20-2}
		If $\vertiii{\cdot}_\psi\in\textbf{{\rm \textbf{N}}}^{\rm{sc}}_{X^n}$, then $ \psi\in\pmb{\Psi}^{\rm{sc}}_n$ and the norm $\|\cdot\|$ is strictly convex.
	\end{enumerate}
\end{theorem}	

\begin{proof}
	Let $\vertiii{\cdot}_\psi\in\textbf{{\rm \textbf{N}}}_{X^n}$ and
	$\mathbf{v}\in\mathbb{S}_X$.\\
	\textbf{Claim 1}: $\psi$ satisfies condition \eqref{S3.2-1}.
	Let $\hat x_i:=(0_X,\ldots,0_X,\textbf{v},0_X,\ldots,0_X)$, where $\mathbf{v}$ is in the $i$th position $(i=1,\ldots,n)$.
	By \eqref{T3.10-1}, we have $\psi(\mathbf{e}_i)=\vertiii{\hat x_i}_\psi=\|\mathbf{v}\|=1$ $(i=1,\ldots,n)$.\\
	\textbf{Claim 2}: $\psi$ satisfies condition \eqref{S3.2-2}.
	Let $t:=(t_1,\ldots,t_{n})\in\Omega^\circ_n$ and $x:=(x_1,\ldots,x_n):=(t_1\textbf{v},\ldots,t_{n}\textbf{v})$.
	For each $i=1,\ldots,n$, by Proposition~\ref{L3.2}\;\eqref{L3.2-1}, we have
\begin{align*}
\psi(t)
&=\vertiii{x}_\psi\ge \vertiii{(t_1\textbf{v},\ldots,t_{i-1}\textbf{v},0_X,t_{i+1}\textbf{v},\ldots,t_n\textbf{v})}_\psi\\ &=(1-t_i)\vertiii{\left(\frac{t_1}{1-t_1}\textbf{v},\ldots,\frac{t_{i-1}}{1-t_i}\textbf{v},0_X,\frac{t_{i+1}}{1-t_i}\textbf{v},\ldots,\frac{t_{n}}{1-t_i}\textbf{v}\right)}_\psi\\
&=(1-t_i)\cdot\psi\left(\frac{t_1}{1-t_i},\ldots,\frac{t_{i-1}}{1-t_i},0,\frac{t_{i+1}}{1-t_i},\ldots,\frac{t_{n}}{1-t_i}\right).
	\end{align*}	
	\textbf{Claim 3}: $\psi$ is continuous.
	Since all norms on $\R^n$ are topologically equivalent, it suffices to show that $\psi$ is continuous with respect to the standard maximum norm on $\R^n$.
	Let $t:=(t_1,\ldots,t_{n})\in\Omega_n$
	and $t^k:=(t^k_1,\ldots,t^k_{n})\in\Omega_n$ for all $k\in\N$.
	Suppose that $t^k\to t$ as $k\to\infty$.
	Let  $x_k:=(t^k_1\textbf{v},\ldots,t^k_n\textbf{v})$ for all $k\in\N$, and $x:=(t_1\textbf{v},\ldots,t_n\textbf{v})$.
	By \eqref{T3.10-1},  we have $\vertiii{x_k}_\psi=\psi(t^k)$ for all $k\in\N$ and $\vertiii{x}_\psi=\psi(t)$.
	By Proposition~\ref{L3.2}\;\eqref{L3.2-2},
	\begin{align*}
		|\psi(t^k)-\psi(t)|
		=\abs{\vertiii{x_k}_\psi-\vertiii{x}_\psi}
		\le\vertiii{x_k-x}_\psi
		\le n\cdot \max_{1\le i\le n}|t^k_i-t_i|\to 0\;\;\text{as}\;\;k\to\infty.
	\end{align*}	
	\textbf{Claim 4}: $\psi$ is convex. 
	Let $t:=(t_1,\ldots,t_{n}), t':=(t'_1,\ldots,t'_{n})\in\Omega_n$, and $\la\in[0,1]$.
	Define
	\begin{gather*}
		x:=(x_1,\ldots,x_n):=\la\cdot(t_1\textbf{v},\ldots,t_n\textbf{v}),\;
		x':=(x'_1,\ldots,x'_n):=(1-\la)\cdot( t'_1\textbf{v},\ldots,t'_n\textbf{v}).
	\end{gather*}
	Then, $\sum_{i=1}^{n}\|x_i\|=\la$, $\sum_{i=1}^{n}\|x'_i\|=1-\la$, and $\sum_{i=1}^{n}\|x_i+x'_i\|=1$.
	By \eqref{T3.10-1},  $\vertiii{x}_\psi=\la\psi(t)$, $\vertiii{x'}_\psi=(1-\la)\psi(t')$ and $\vertiii{x+x'}_\psi=\psi(\la t+(1-\la)t')$.
	By the triangle inequality of $\vertiii{\cdot}_\psi$, we obtain
	$\psi(\la t+(1-\la)t')\le \la\psi(t)+(1-\la)\psi(t')$.
	Thus, assertion \eqref{P3.20-1} is proved.
	
	Assume that $\vertiii{\cdot}_{\psi}\in\textbf{{\rm \textbf{N}}}^{\rm{sc}}_{X^n}$.
	Using similar arguments as in the proof of \textbf{Claim 4}, we obtain that 
	$\psi$ is strictly convex.
	It remains to show that $\|\cdot\|$ is strictly convex.
	Let $\mathbf{v},\mathbf{v}'\in\mathbb{S}_X$ be distinct vectors, $x:=(\mathbf{v},0_X,\ldots,0_X)$ and $x':=(\mathbf{v}',0_X,\ldots,0_X)$.
	Thus, $x,x'\in\mathbb{S}_{X^n}$ and $x\ne x'$.
	By condition \eqref{A2} and the strict convexity of $\vertiii{\cdot}_\psi$,  we obtain
	$
	\|\mathbf{v}+\mathbf{v}'\|=\vertiii{(\mathbf{v}+\mathbf{v}',0_X,\ldots,0_X)}_\psi=\vertiii{x+x'}_\psi<2.
	$
	This completes the proof.
\end{proof}	

The next statement is a direct consequence of Theorem~\ref{T3.10}\;\eqref{T3.10.1} and Theorem~\ref{P3.10}\;\eqref{P3.20-2}.
\begin{corollary}\label{C2.9}
	Let $\psi:\Omega_n\to\R$, and the function $\vertiii{\cdot}_\psi$ be defined by \eqref{T3.10-1}.
	Then $\vertiii{\cdot}_\psi\in \textbf{{\rm \textbf{N}}}^{\rm{sc}}_{X^n}$ if and only if $\psi\in \pmb{\Psi}^{\rm{sc}}_{n}$ and  $\|\cdot\|$ is strictly convex.
\end{corollary}	

\begin{remark}
	A result similar to Corollary~\ref{C2.9} in the setting of direct sums of Banach spaces can be found in \cite[Theorem~3.3]{KatSaiTam03}.
\end{remark}	

\section{Dual norms}\label{S4}

The following theorem provides an explicit expression for the dual norm of a given primal norm of the form \eqref{T3.10-1}.
\begin{theorem}\label{T4.2}
	Let $\psi\in\pmb{\Psi}_n$.
	Then
		\begin{gather}\label{T4.2-1}
			\vertiii{(x^*_1,\ldots,x^*_n)}_{\psi} =\max_{(t_1,\ldots,t_n)\in\Omega_n}\dfrac{\sum_{i=1}^nt_i\|x^*_i\|}{\psi (t_1,\ldots,t_{n})}.
		\end{gather}	
\end{theorem}	
\begin{proof}
	Let $\vertiii{\cdot}_\psi$ be given by \eqref{T3.10-1}, and $x^*:=(x^*_1,\ldots,x^*_n)\in(X^n)^*$.
	Let $x:=(x_1,\ldots,x_n)\in X^n$ with $\vertiii{x}_\psi=1$,  and $t:=(t_1,\ldots,t_n):=\left(\frac{\|x_1\|}{\sum_{i=1}^{n}\|x_i\|},\ldots,\frac{\|x_n\|}{\sum_{i=1}^{n}\|x_i\|}\right).$
	Then $t\in\Omega_n$ and $\al\psi(t)=1$.
	We have
	\begin{align*}
		\vertiii{x^*}_{\psi}
		=\sup_{\vertiii{x}_\psi=1}
		\sum_{i=1}^{n}\langle x^*_i,x_i\rangle
		&=\sup_{\vertiii{x}_\psi=1}
		\sum_{i=1}^{n}\|x^*_i\|\cdot\|x_i\|
		=\max_{(t_1,\ldots,t_n)\in\Omega_n}\dfrac{\sum_{i=1}^nt_i\|x^*_i\|}{\psi (t_1,\ldots,t_{n})},
	\end{align*}	
	where the last equality stems from the fact that the function
	$(s_1,\ldots,s_n)\mapsto \frac{\sum_{i=1}^ns_i\|x^*_i\|}{\psi (s_1,\ldots,s_{n})}$
	is continuous (with respect to the topology induced by any chosen norm on $\R^n$) on the compact set $\Omega_n$.
	This completes the proof.
\end{proof}

\begin{remark}
	The dual space $(X^n)^*$ is isomorphic to the product space $(X^*)^n$, where $X$ is a given vector space.
	Indeed, one can verify that the mapping $L:(X^*)^n\to (X^n)^*$, defined by $\langle L(x^*),x\rangle:=\sum_{i=1}^{n}\langle x^*_i,x_i\rangle$
	for all $x^*:=(x^*_1,\ldots,x^*_n)\in (X^*)^n$ and $x:=(x_1,\ldots,x_n)\in X^n$, is linear and bijective; see  \cite[Proposition~3.4]{MitOshSai05} and \cite[Theorem~2]{SomAttSat05}.
	Therefore, for each $x^*\in (X^n)^*$, there exist  unique vectors     $x^*_1,\ldots,x^*_n\in X^*$ such that $x^*=(x^*_1,\ldots,x^*_n)$. 	
\end{remark}	

For each $\psi\in\pmb{\Psi}_n$, define the function $\psi^*:\Omega_n\to\R$ by 
	\begin{gather}\label{psi1}
		\psi^*(s):=\max_{t\in\Omega_n}\dfrac{\langle t,s\rangle}{\psi(t)}
		\;\;
		\text{for all}
		\;\; s\in\Omega_n.
	\end{gather}	
This function was originally introduced in \cite{MitOshSai05,SomAttSat05}.
Observe that the dual  norm of $\vertiii{\cdot}_\psi$ on $X^n$ is the norm $\vertiii{\cdot}_{\psi^*}$ on $(X^n)^*$.
We call $\psi^*$ the dual function of $\psi$. 
It follows directly from \eqref{psi1} that
$\langle t,s\rangle\le\psi(t)\psi^*(s)$ for all $t,s\in\Omega_n$.
The next proposition collects some basic properties of $\psi^*$.
\begin{proposition}\label{P3.2}
	Let $\psi\in \pmb{\Psi}_n$.
	The following assertions holds:
	\begin{enumerate}[\rm (i)]
		\item\label{P3.2-1}
		$\psi^*\in\pmb{\Psi}_n$; 	
		\item\label{P3.2-2}
		$\vertiii{\cdot}^*_\psi\in\textbf{{\rm \textbf{N}}}^{\rm{sc}}_{(X^n)^*}$ if and only if $ \psi^*\in\pmb{\Psi}^{\rm{sc}}_n$ and $\|\cdot\|_{X^*}$ is strictly convex;
		\item\label{P3.2-3}
		$\psi^*(s)=\vertiii{(s_1\mathbf{u}^*_1,\ldots,s_{n}\mathbf{u}^*_n)}_{\psi^*}$ for all $(\mathbf{u}^*_1,\ldots,\mathbf{u}^*_n)\in \mathbb{S}^n_{X^*}$ and $s:=(s_1,\ldots,s_{n})\in\Omega_n$;
		\item\label{P3.2-4}
		$\langle \bar t,\bar s\rangle=\psi(\bar t)\psi^*(\bar s)$ for some $\bar t,\bar s\in\Omega_n$  if and only if $\frac{ \bar s}{\vertiii{\bar s}}_{\psi^*}\in \partial\vertiii{\cdot}_{\psi} (\bar t)$;
		\item\label{P3.2-5}
		$\psi(t)=\max\limits_{s\in\Omega_n}\frac{\langle t,s\rangle}{\psi^*(s)}$
		for all $t\in\Omega_n$.
	\end{enumerate}	
\end{proposition}	
\begin{proof}
	It follows from \eqref{T4.2-1} that
	$\vertiii{\cdot}_{\psi^*}\in\textbf{{\rm \textbf{N}}}_{(X^n)^*}$.
	By Theorem~\ref{P3.10}\;\eqref{P3.20-1} with  $\psi^*$ in place of $\psi$, we obtain $\psi^*\in\pmb{\Psi}_n$.
	Thus, assertion \eqref{P3.2-1} is proved.
	Parts \eqref{P3.2-2} and \eqref{P3.2-3} are dual counterparts of Corollary~\ref{C2.9} and Theorem~\ref{T3.10}\;\eqref{T3.10.3}, respectively.
	The proofs are similar to those of these results and therefore are skipped.
	
	Let $\xi:=\frac{\bar s}{\vertiii{\bar s}_{\psi^*}}$.
	Then $\vertiii{\xi}_{\psi^*}=1$.
	Note that $\psi(\bar t)=\vertiii{\bar t}_\psi$ and $\psi^*(\bar s)=\vertiii{\bar s}_{\psi^*}$.
	If $\langle \bar t,\bar s\rangle=\psi(\bar t)\psi^*(\bar s)$, then dividing both sides of the latter by $\psi^*(\bar s)$, we have
	$\left\langle \xi,\bar t\right\rangle=\vertiii{\bar t}_{\psi}$.
	Thus, $\xi\in\partial\vertiii{\cdot}_{\psi}(\bar t)$.
	Conversely, if $\xi\in\partial\vertiii{\cdot}_{\psi}(\bar t)$, then $\left\langle \xi,\bar t\right\rangle=\vertiii{\bar t}_{\psi}$, or equivalently,
	$\langle \bar s,\bar t\rangle=\vertiii{\bar s}_{\psi^*}\cdot\vertiii{\bar t}_{\psi}=\psi(\bar t)\psi^*(\bar s)$.
	Thus, assertion \eqref{P3.2-4} is proved.
	
	We now prove \eqref{P3.2-5}.
	We first show that 
	\begin{gather}\label{P4.11.6}
		\vertiii{x}_\psi=\sup_{\vertiii{x^*}_{\psi^*}=1}\langle x^*,x\rangle\;\text{for all}\;x\in X^n.
	\end{gather}	
	If $x=0_{X^n}$, then condition \eqref{P4.11.6} is trivially satisfied.
	Let $x\ne 0_{X^n}$ and
	$x^*\in(X^n)^*$ with $\vertiii{x^*}_{\psi^*}=1$.
	Then $\langle x^*,x\rangle\le \vertiii{x^*}_{\psi^*}\cdot  \vertiii{x}_\psi=\vertiii{x}_\psi$, and consequently,
	$\vertiii{x}_\psi\ge \sup_{\vertiii{x^*}_{\psi^*}=1}\langle x^*,x\rangle.$
	By the Hahn-Banach theorem, there exists an $\hat x^*\in (X^n)^*$ such that $\vertiii{\hat x^*}_{\psi^*}=1$ and $\langle\hat x^*,x\rangle=\vertiii{x}_{\psi}$.
	Thus, $\vertiii{x}_{\psi}\le \sup_{\vertiii{x^*}_{\psi^*}=1}\langle x^*,x\rangle.$
	Therefore, condition \eqref{P4.11.6} is satisfied.
	Let $x^*:=(x^*_1,\ldots,x^*_n)\in(X^n)^*$ and
	$s_i:=\frac{\|x^*_i\|}{\sum_{i=1}^{n}\|x^*_i\|}$ 
	$(i=1,\ldots,n)$.
	By \eqref{P4.11.6},
	\begin{align*}
		\vertiii{x}_{\psi}
		&=\sup_{\vertiii{x}_{\psi^*}=1}
		\sum_{i=1}^{n}\langle x^*_i,x_i\rangle
		=\sup_{\vertiii{x}_{\psi^*}=1}
		\sum_{i=1}^{n}\|x^*_i\|\cdot\|x_i\|
		=\max_{(s_1,\ldots,s_n)\in\Omega_n}\dfrac{\sum_{i=1}^ns_i\|x_i\|}{\psi^*(s_1,\ldots,s_{n})}.
	\end{align*}	
	This completes the proof.
\end{proof}	

%The following proposition is important.
\begin{example}\label{E4.8}
	Let	$\psi_p$ be given by \eqref{pf} for some
	$p\in[1,\infty]$, and  $q\ge 1$ satisfy $\frac{1}{p}+\frac{1}{q}=1$.
	Then
		\begin{equation*}
			\psi^*_{p}(s)
			= \begin{cases}
				\max\{s_1,\ldots,s_n\}	& \text{if } p=1,\\
				\left(s_1^q+\ldots+s_n^q\right)^{\frac{1}{q}}  & \text{if } p\in(1,\infty),\\
				1 & \text{if } p=\infty
			\end{cases} 
		\end{equation*}
	for all $s:=(s_1,\ldots,s_{n})\in\Omega_n$.
\end{example}	

Let $\psi,\phi\in \pmb{\Psi}_n$, and $\psi^*$ and $\phi^*$ be the corresponding dual functions.
Define 
\begin{gather}\label{mM}
	m:=\min_{t\in\Omega_n}\dfrac{\psi(t)}{\phi(t)},\;\; M:=\max_{t\in\Omega_n}\dfrac{\psi(t)}{\phi(t)},\;\;
	m^*:=\min_{s\in\Omega_n}\dfrac{\psi^*(s)}{\phi^*(s)},\;\; M^*:=\max_{s\in\Omega_n}\dfrac{\psi^*(s)}{\phi^*(s)}.
\end{gather}	
Hereafter, the notations $\psi\ge\phi$ and $\psi^*\ge\phi^*$  mean that $\psi(t)\ge\phi(t)$ and  $\psi^*(t)\ge\phi^*(t)$ for all $t\in\Omega_n$, respectively.
\begin{proposition}\label{P7.1}
	Let $\psi,\phi\in \pmb{\Psi}_n$.
	The following statements hold:
	\begin{enumerate}[\rm (i)]
		\item\label{P7.1-0}
		$\frac{1}{n}\le m\le \frac{\psi(t)}{\phi(t)} \le M\le n$ for all $t\in\Omega_n$;
		\item\label{P7.1-1}
		$\frac{1}{n}\le m^*\le \frac{\psi^*(t)}{\phi^*(t)}\le M^*\le n$ for all $t\in\Omega_n$;
		\item\label{P7.1-3}
		$m\cdot M^*= M\cdot m^*=1$;
		\item\label{P7.1-6}
		$m\cdot\vertiii{\cdot}_{\phi}\le\vertiii{\cdot}_{\psi}\le M\cdot\vertiii{\cdot}_{\phi}$;
		\item\label{P7.1-7}
		$\frac{1}{M}\cdot\vertiii{\cdot}_{\phi^*}\le\vertiii{\cdot}_{\psi^*}\le\frac{1}{m}\cdot\vertiii{\cdot}_{\phi^*}$;
		\item\label{P7.1-4}
		$\psi\ge\phi$ if and only if $\psi^*\le\phi^*$;
		\item \label{P7.1-5}
		$\psi\le\phi$ if and only if $\psi^*\ge\phi^*$.
	\end{enumerate}		
\end{proposition}	

\begin{proof}
	By Proposition~\ref{P3.2}\;\eqref{P3.2-1}, we have $\psi^*,\phi^*\in \pmb{\Psi}_n$.
	Parts \eqref{P7.1-0} and \eqref{P7.1-1}  are direct consequences of Proposition~\ref{P5.1}.
	The estimates
	\begin{gather*}
		\psi^*(s)=\max_{t\in\Omega_n}\frac{\langle t,s\rangle}{\psi(t)}\le\dfrac{1}{m}\cdot \max_{t\in\Omega_n}\frac{ \langle t,s\rangle}{\phi(t)}=\frac{1}{m}\cdot\phi^*(s)\;\text{for all}\;s\in\Omega_n,\\
		\psi(t)=\max_{s\in\Omega_n}\frac{\langle s,t\rangle}{\psi^*(s)}\ge\dfrac{1}{M^*}\cdot\max_{s\in\Omega_n}\dfrac{\langle s,t\rangle}{\phi^*(s)}=\dfrac{1}{M^*}\cdot\phi(t)\;\text{for all}\;
		t\in\Omega_n
	\end{gather*}	
	imply $M^*\le\frac{1}{m}$ and $M^*\ge\frac{1}{m}$, respectively.
	Thus, $m\cdot M^*=1$.
	The proof of the equality $M\cdot m^*=1$ is similar.
	Hence, assertion \eqref{P7.1-3} is proved.
	In view of \eqref{T3.10-1}, assertions \eqref{P7.1-6} and \eqref{P7.1-7} are consequences of \eqref{P7.1-0}-\eqref{P7.1-3}.
	If  $\psi\ge \phi$, then $\psi^*(s)=\max\limits_{t\in\Omega_n}\frac{\langle t,s\rangle}{\psi(t)}\le \max\limits_{t\in\Omega_n}\frac{ \langle t,s\rangle}{\phi(t)}=\phi^*(s)$ for all $s\in\Omega_n$.
	Conversely, if $\psi^*\ge\phi^*$, then $\psi(t)=\max\limits_{s\in\Omega_n}\frac{\langle s,t\rangle}{\psi^*(s)}\ge\max\limits_{s\in\Omega_n}\frac{\langle s,t\rangle}{\phi^*(s)}=\phi(t)$ for all $t\in\Omega_n$.
	Thus, assertion \eqref{P7.1-4} is proved. 
	The proof of \eqref{P7.1-5} is similar.
\end{proof}	

\begin{remark}
	Part \eqref{P7.1-0} of Proposition~\ref{P7.1} recaptures 	 \cite[Theorem~2.1]{Cao03}.
\end{remark}	

\section{Subdifferential characterizations}\label{S5}
In this section, we provide an explicit formula for the convex subdifferential of the norm of the form
\eqref{T3.10-1}.

The next statement is essential.
\begin{lemma}\label{L5.1}
	Let $\psi\in\pmb{\Psi}_n$, and $\nu:=(\nu_1,\ldots,\nu_n)\in\R^n$ with $\vertiii{\nu}_\psi=1$.
	Then
		\begin{align}\notag
			\partial\vertiii{\cdot}_{\psi}(\nu)=\big\{(
			&\tau_1\gamma_1
			,\ldots,\tau_n\ga_n)\in\R^n \mid \mu\in\partial\psi(t),\;\ga_i:=\psi(t)+\langle \mu,\mathbf{e}_i-t\rangle\ge 0,\\ \label{L4.12-1}
			&\tau_i:=\text{\rm sgn}(\nu_i)\;\text{\rm if}\;\nu_i\ne 0\;\text{\rm and}\;\tau_i\in\{-1,1\}\;\text{\rm if}\;\nu_i=0\; (i=1,\ldots,n)\},
		\end{align}
	where $t:=\left(\frac{|\nu_1|}{\sum_{i=1}^{n}|\nu_i|},\ldots,\frac{|\nu_n|}{\sum_{i=1}^{n}\abs{\nu_i}}\right)$.
\end{lemma}	

\begin{proof}
	Let $t:=(t_1,\ldots
	,t_n):=\left(\frac{|\nu_1|}{\sum_{i=1}^{n}|\nu_i|},\ldots,\frac{|\nu_n|}{\sum_{i=1}^{n}\abs{\nu_i}}\right)$, and $A$ denote the set in the right-hand side of \eqref{L4.12-1}. 
	We first prove that $\partial\vertiii{\cdot}_\psi(\nu)\subset A$.
	Let $\xi:=(\xi_1,\ldots,\xi_n)\in\partial\vertiii{\cdot}_\psi(\nu)$.
	Then  $\vertiii{\xi}_{\psi^*} =1$, $\langle \xi,\nu\rangle=1$, and $\sum_{i=1}^{n}\abs{\nu_i}=\frac{1}{\psi(t)}$.
	Thus, $\abs{\nu_i}=\frac{t_i}{\psi(t)}$ $(i=1,\ldots,n)$.
	Set $\bar t:=(t_1\cdot\text{sgn}(\nu_1),\ldots,t_n\cdot\text{sgn}(\nu_n))$.
	We have $\nu=\frac{\bar t}{\psi(t)}$, and consequently, $\psi(t)=\langle\xi,\bar t\rangle$.
	Let  $\mu:=(\abs{\xi_1},\ldots,\abs{\xi_n})$ and
	$\xi':=\frac{\mu}{\sum_{i=1}^{n}\abs{\xi_i}}$.
	Then $\psi(t)=\langle\xi,\bar t\rangle
	\le\langle\mu,t\rangle\le \sum_{i=1}^{n}\abs{\xi_i}\cdot\psi^*(\xi')\cdot\psi(t)=
	\vertiii{\xi}_{\psi^*}\cdot\psi(t)=\psi(t).$
	This implies $\psi(t)=\langle \mu,t\rangle$ and $\xi_it_i\cdot\text{sgn}(\nu_i)=|\xi_i|t_i\;\;(i=1,\ldots,n).$
	Consequently, $\xi_i=\abs{\xi_i}\cdot\text{sgn}(\nu_i)$ whenever $\nu_i\ne0$ $(i=1,\ldots,n)$.
	Let $t':=(t'_1,\ldots,t'_n)\in\Omega_n$.
	We have
	\begin{align*}
		\psi(t')=\vertiii{t'}_{\psi}&=\vertiii{\xi}_{\psi*}\cdot\vertiii{t'}_\psi=\vertiii{\mu}_{\psi*}\cdot\vertiii{t'}_\psi\\
		&\ge\abs{\langle\mu,t'\rangle}=
		\langle\mu,t'\rangle
		=\langle\mu,t'\rangle+\psi(t)-\langle\mu,t\rangle
		=\psi(t)+\langle \mu,t'-t\rangle.
	\end{align*}	
	Therefore, $\mu\in\partial\psi(t)$.
	Let
	$
	\ga_i:=\psi(t)+\langle \mu,\mathbf{e}_i-t\rangle
	=\psi(t)-\langle\mu,t\rangle+\abs{\xi_i}=\abs{\xi_i}\ge 0\;\;(i=1,\ldots,n).
	$
	Then $\xi_i=\tau_i\ga_i$ with $\tau_i:=\text{\rm sgn}(\nu_i)$ if $\nu_i\ne 0$, and $\xi_i=\tau_i\ga_i$ with $\tau_i\in\{-1,1\}$ if $\nu_i=0$ $(i=1,\ldots,n)$.
	Thus, $\partial\vertiii{\cdot}_{\psi}(\nu)\subset A$.
	
	Conversely, let $\xi:=(\xi_1,\ldots,\xi_n)\in A$.
	Then there exists a vector $\mu:=(\mu_1,\ldots,\mu_n)\in\partial\psi(t)$ such that 
	$\ga_i:=\psi(t)+\langle \mu,\mathbf{e}_i-t\rangle\ge 0$, $\xi_i=\tau_i\ga_i$ with $\tau_i:=\text{sgn}(\nu_i)$ if $\nu_i\ne0$ and $\tau_i\in\{-1,1\}$ if $\nu_i=0$ $(i=1,\ldots,n)$. 
	We have $\abs{\xi_i}=\psi(t)+\langle \mu,\mathbf{e}_i-t\rangle$ $(i=1,\ldots,n)$.
	Recall that $\sum_{i=1}^{n}{t_i}=1$.
	We have
	$\sum_{i=1}^{n}\abs{\xi_i}t_i
	=\sum_{i=1}^{n}t_i\cdot\left(\psi(t)+\langle \mu,\mathbf{e}_i-t\rangle\right)
	=\psi(t)+\sum_{i=1}^{n}{t_i}\mu_i-\langle \mu,t\rangle=\psi(t).
	$
	Note that $\xi_i\nu_i\ge 0$ for all $i=1,\ldots,n$.
	Thus,
	\begin{align*}
		\langle \xi,\nu\rangle
		=\sum_{i=1}^{n}\abs{\xi_i}\abs{\nu_i}
		=\left(\sum_{i=1}^{n}\abs{\nu_i}\right)\cdot\sum_{i=1}^{n}\abs{\xi_i}\abs{t_i}
		=\left(\sum_{i=1}^{n}\abs{\nu_i}\right)\cdot\psi(t)=\vertiii{\nu}_{\psi}=1.
	\end{align*}	 
	This implies $\vertiii{\xi}_{\psi^*}=1$, and consequently, $\xi\in\partial\vertiii{\cdot}_{\psi}(\nu)$ meaning $A\subset\partial\vertiii{\cdot}_{\psi}(\nu)$.
	The proof is complete.
\end{proof}	

\begin{remark}
	Lemma~\ref{L5.1} is a symmetric counterpart of \cite[Corollary~4.5]{MitSaiSuz03} in the real setting, where the complex space $\mathbb{C}^n$ is restricted to its real subspace $\mathbb{R}^n$.
\end{remark}	

\begin{theorem}\label{T5.2}
	Let $\psi\in\pmb{\Psi}_n$, and $x:=(x_1,\ldots,x_n)\in X^n$ with $\vertiii{x}_{\psi}=1$.
	Then
		\begin{gather}\label{T4.12-5}
			\partial\vertiii{\cdot}_\psi(x) = \left\{(\xi_1 x^*_1,\ldots,\xi_n x^*_n)\mid x^*_i\in\partial\|\cdot\|(x_i)\;(i=1,\ldots,n),\;(\xi_1,\ldots,\xi_n)\in\partial\vertiii{\cdot}_{\psi}(\nu)\right\},
		\end{gather}	
	where $\nu:=(\|x_1\|,\ldots,\|x_n\|)$.
\end{theorem}	
\begin{proof}
	Let $\nu:=(\nu_1,\ldots,\nu_n):=(\|x_1\|,\ldots,\|x_n\|)$, and  
	$B$ denote the set in the right-hand side of \eqref{T4.12-5}.
	We first prove that $\partial\vertiii{\cdot}_\psi(x)\subset B$.
	Let $u^*:=(u^*_1,\ldots,u^*_n)\in\partial\vertiii{\cdot}_\psi(x)$.
	Then $\vertiii{u^*}_{\psi^*}=1$ and $\langle u^*,x\rangle=1$.
	Set $\xi:=(\xi_1,\ldots,\xi_n):=(\|u^*_1\|,\ldots,\|u^*_n\|)$,
	$\al:=\sum_{i=1}^{n}\nu_i$, $\be:=\sum_{i=1}^{n}\xi_i$, $t:=\frac{\nu}{\al}$, $s:=\frac{\xi}{\be}.$
	Then $\al\psi(t)=\vertiii{x}_{\psi}=\vertiii{\nu}_{\psi}=1$ and 
	$\be\psi^*(s)=\vertiii{u^*}_{\psi^*}=\vertiii{\xi}_{\psi^*}=1.$
	We have
	$
	1=\langle u^*,x\rangle
	\le\langle \xi,\nu\rangle
	=\al\be\langle s,t\rangle\le\al\psi(t) \cdot\be\psi^*(s)=1.
	$
	This implies $\langle \xi,\nu\rangle=1$ and $ \langle u^*_i,x_i\rangle=\xi_i\cdot\|x_i\|$ $(i=1,\ldots,n)$.
	Consequently,
	$\xi\in\partial\vertiii{\cdot}_{\psi}(\nu)$.
	For each $i=1,\ldots,n$, choose a dual vector $\hat x^*_i\in\partial \|\cdot\|(x_i)$, and define $x^*_i:=\frac{u^*_i}{\xi_i}$ if  $\xi_i\ne 0$ and $x^*_i:=\hat x^*_i$ if  $\xi_i= 0$.
	Then $x ^*_i\in\partial\|\cdot\|(x_i)$ $(i=1,\ldots,n)$ and $u^*=(\xi_1 x^*_1,\ldots,\xi_n x^*_n)$.
	Thus, $u^*\in B$ meaning $\partial\vertiii{\cdot}_\psi(x)\subset B$.
	
	Conversely, let $u^*:=(u^*_1,\ldots,u^*_n)\in B$.
	Then there exist vectors $x^*_i\in\partial\|\cdot\|(x_i)$ and $\xi:=(\xi_1,\ldots,\xi_n)\in\partial\vertiii{\cdot}_{\psi}(\nu)$ such that
	$u^*_i=\xi_i x^*_i$ $(i=1,\ldots,n)$.
	We have $\vertiii{\xi}_{\psi^*}=1$ and $\langle \xi,\nu\rangle=\vertiii{\nu}_\psi.$
	Consequently, 
	$\langle u^*,x\rangle=\sum_{i=1}^{n}\xi_i\langle x^*_i,x_i\rangle
	=\sum_{i=1}^{n}\xi_i\|x_i\|=\langle\xi,\nu\rangle=\vertiii{\nu}_\psi=\vertiii{x}_\psi.$
	Thus, $u^*\in\partial\vertiii{\cdot}_{\psi}(x)$ meaning $\partial\vertiii{\cdot}_\psi(x)\subset B$.
	This completes the proof.
\end{proof}	

\section{Extremal principle}\label{S6}
In this section, we assume that $\Omega_1,\ldots,\Omega_n$ are nonempty subsets of a vector space $X$, $\bx\in\bigcap_{i=1}^n\Omega_i$, and $\widehat{\Omega}:=\Omega_1\times\ldots\times\Omega_n$.

The following extremality property has been studied recently in \cite{CuoKru24-2}.
\begin{definition}\label{D5.1}
	Let $X$ be a vector space.
	The collection $\{\Omega_1,\ldots,\Omega_n\}$ is 
	extremal at $\bx$ with respect to $\vertiii{\cdot}_{X^n}$ if there is a $\rho>0$ such that,
	for any $\varepsilon>0$, there exists a point $(a_1,\ldots,a_n)\in\eps\B_{X^n}$ such that 
	$\bigcap_{i=1}^n(\Omega_i-a_i)\cap B_\rho(\bx)=\emptyset.$
\end{definition}

The above definition assumes that $X$ is a vector space and employs only the norm structure of $X^n$.
If $X$ is equipped with a norm $\|\cdot\|$ and
$\vertiii{(x_1,\ldots,x_n)}_{X^n}:=\max_{1\le i\le n}\|x_i\|$ for all
$(x_1,\ldots,x_n)\in X^n$,
then it reduces to the conventional extremality \cite{KruMor80}.
The extremal principle, a dual necessary condition for the extremality,
serves as a powerful tool for deriving optimality conditions in nonconvex and nonsmooth optimization problems, as well as for establishing subdifferential and coderivative calculus rules \cite{Mor06.1}.
In \cite{CuoKru24-2}, the authors have established several generalized separation results for collections of sets, from which the conventional extremal principle follows.
The study in \cite{CuoKru24-2}  employs several compatibility conditions between the norms on product spaces.
Let $\|\cdot\|$, $\vertiii{\cdot}_{X^{n-1}}$ and $\vertiii{\cdot}_{X^n}$ be norms on $X$, $X^{n-1}$ and $X^n$, respectively. 
We consider the following compatibility conditions:
{\renewcommand\theenumi{(C\arabic{enumi})}
	\renewcommand\labelenumi{\theenumi} 
	\begin{enumerate}
		\item\label{C1}
		$\max\{\vertiii{(x_1,\ldots,x_{n-1})}_{X^{n-1}}, \vertiii{(x_n,\ldots,x_{n})}_{X^{n-1}}\}\le \kappa\vertiii{x}_{X^n}$ for all $x:=(x_1,\ldots,x_n)\in X^n$ and some $\kappa>0$;
		\item\label{C2}
		$\max\{\vertiii{(x_1,\ldots,x_{n-1})}_{X^{n-1}}, \|x_{n}\|\}\le \kappa\vertiii{x}_{X^n}$ for all $x:=(x_1,\ldots,x_n)\in X^n$ and some $\kappa>0$;
		\item\label{C3}
		$\max\{\|x_1\|,\ldots,\|x_n\|\}\le\kappa\vertiii{x}_{X^n}$ for all $x:=(x_1,\ldots,x_n)\in X^n$ and some $\kappa>0$;
		\item\label{C4}
		$\vertiii{x}_{X^n}\le\kappa\max\{\|x_1\|,\ldots,\|x_n\|\}$ for all $u\in X$ and some $\kappa>0$;
		\item\label{C5}
		$\vertiii{(u,\ldots,u)}_{X^n}\le\kappa\|u\|$;
		\item\label{C6}
		$\sum_{i=1}^n\|x_i^*\|\le \kappa\vertiii{x^*}_{(X^n)^*}$ for all $x^*:=(x^*_1,\ldots,x^*_n)\in(X^n)^*$ and some $\kappa>0$;
		\item\label{C7}
		$\vertiii{x^*}_{(X^n)^*}\le\kappa\cdot \sum_{i=1}^n\|x_i^*\|$  for all $x^*:=(x^*_1,\ldots,x^*_n)\in(X^n)^*$ and some $\kappa>0$.
	\end{enumerate}
} 
Detailed discussions of relations between these conditions can be found in \cite{CuoKru24-2}.
The following theorem, established recently in \cite{CuoKru24-2}, provides a dual necessary condition for the extremality in Definition~\ref{D5.1} in terms of Fr\'echet normals in Asplund spaces.
\begin{theorem}\label{C6.4}
	Let $(X,\|\cdot\|)$ and $(X^n,\vertiii{\cdot})$ be Asplund, $\widehat\Omega$ be closed, conditions {\rm\ref{C5}} and {\rm\ref{C6}} be satisfied.
	If $\{\Omega_1,\ldots,\Omega_n\}$ is extremal at $\bx$ with respect to $\vertiii{\cdot}$, then, for any $\varepsilon>0$, there exist $ x:=(x_1,\ldots,x_n)\in\widehat\Omega\cap B_{\varepsilon}(\bx,\ldots,\bx)$ and $(x_1^*,\ldots,x_n^*)\in N^F_{\widehat\Omega}(x)$ such that 
	$\|\sum_{i=1}^nx_i^*\|<\eps$ and $
	\vertiii{(x_1^*,\ldots,x_n^*)}=1.
	$
\end{theorem}

The above theorem employs conditions \ref{C5} and \ref{C6}.
For the application of the other  compatibility conditions in generalized separations results for collections of sets, the reader is referred to \cite{CuoKru24-2}.
In the following statements, we show that the class of sign-symmetric norms defined by \eqref{T3.10-1} satisfies all the compatibility conditions  \ref{C1}-\ref{C7}.
\begin{lemma}\label{P6.6}
	Let $\phi\in\pmb{\Psi}_n$.
	Define 
		\begin{gather}\label{varphi}
			\bar \phi(t):=\phi(t_1,\ldots,t_{n-1},0)\;\;\text{for all}\;\; t:=(t_1,\ldots,t_{n-1})\in\Omega_{n-1}.
		\end{gather}	
	Then, $\bar \phi\in\pmb{\Psi}_{n-1}$.
\end{lemma}	
\begin{proof}
	Since $\phi$ satisfies \eqref{S3.2-1}, we have
	$\bar\phi\left(1,0,\ldots,0\right)=\phi(\mathbf{e}_1)=1,\ldots,
	\bar\phi\left(0,0,\ldots,1\right)=\phi(\mathbf{e}_{n-1})=1.$
	For each $i=1,\ldots,n-1$, by \eqref{S3.2-2}, we have
	\begin{align*}
		\bar\phi(t)=\phi(t_1,\ldots,t_{n-1},0)
		&\ge (1-t_i)\cdot\phi\left(\dfrac{t_1}{1-t_i},\ldots,\dfrac{t_{i-1}}{1-t_i},0,\dfrac{t_{i+1}}{1-t_i},\ldots,\dfrac{t_{n-1}}{1-t_i},0\right)\\
		&=(1-t_i)\cdot\bar\phi\left(\dfrac{t_1}{1-t_i},\ldots,\dfrac{t_{i-1}}{1-t_i},0,\dfrac{t_{i+1}}{1-t_i},\ldots,\dfrac{t_{n-1}}{1-t_i}\right)
	\end{align*}
	for all $t:=(t_1,\ldots,t_{n-1})\in\Omega^\circ_{n-1}$.
	Thus, $\bar \phi\in\pmb{\Psi}_{n-1}$.
\end{proof}	

\begin{proposition}\label{P6.5}
	Let $\phi\in\pmb{\Psi}_n$, $\bar\phi$ be defined by \eqref{varphi}, $u\in X$, $x:=(x_1,\ldots,x_n)\in X^n$ and $x^*:=(x^*_1,\ldots,x^*_n)\in(X^n)^*$.
	The following inequalities hold:
	\begin{enumerate}[\rm (i)]
		\item\label{P6.5-1}
		$\max\{\vertiii{(x_1,\ldots,x_{n-1})}_{\bar\phi}, \vertiii{(x_n,\ldots,x_{n})}_{\bar\phi}\}\le (n-1)\cdot\vertiii{x}_\phi$;
		\item\label{P6.5-2}
		$\max\{\vertiii{(x_1,\ldots,x_{n-1})}_{\bar\phi}, \|x_{n}\|\}\le \vertiii{x}_\phi$;
		\item\label{P6.5-3}
		$\max\{\|x_1\|,\ldots,\|x_n\|\}\le\vertiii{x}_{\phi}$;
		\item\label{P6.5-4}
		$\vertiii{x}_{\phi}\le n\cdot\max\{\|x_1\|,\ldots,\|x_n\|\}$;
		\item\label{P6.5-5}
		$\vertiii{(u,\ldots,u)}_{\phi}\le n\cdot\|u\|$;
		\item\label{P6.5-6}
		$\sum_{i=1}^n\|x_i^*\|\le n\cdot\vertiii{x^*}_{\phi^*}$;
		\item\label{P6.5-7}
		$\vertiii{x^*}_{\phi^*}\le \sum_{i=1}^n\|x_i^*\|$.
	\end{enumerate}
\end{proposition}	

\begin{proof}
	Observe that
	\begin{gather}\label{P5.4-1}
		\vertiii{(x_1,\ldots,x_{n-1})}_{\bar\phi}=\vertiii{(x_1,\ldots,x_{n-1},0_X)}_\phi
		\;\text{for all}\;(x_1,\ldots,x_{n-1})\in X^{n-1}.
	\end{gather}	
	Thus,  \eqref{P6.5-1} and \eqref{P6.5-2} follow from Proposition~\ref{L3.2}\;\eqref{L3.2-0} and property \eqref{A2} of $\vertiii{\cdot}_\phi$, while 
	\eqref{P6.5-3} and \eqref{P6.5-4} are consequences of
	Proposition~\ref{L3.2}\;\eqref{L3.2-2}.
	Condition \eqref{A2} clearly yields \eqref{P6.5-5}.
	Finally,   \eqref{P6.5-6} and \eqref{P6.5-7} follow from Proposition~\ref{L3.2}\;\eqref{L3.2-2} and Proposition~\ref{P7.1}\;\eqref{P7.1-6} and \eqref{P7.1-7}.
\end{proof}	
\if{
	By Proposition~\ref{P6.5}\;\eqref{P6.5-1} and \eqref{P6.5-2}, we can choose a norm on $X^{n-1}$
	that has the same structure as the norm on  $X^n$ and satisfies conditions \eqref{C1} and \eqref{C2}.
}\fi
\begin{proposition}\label{C6.5}
	Let $\psi,\phi\in\pmb{\Psi}_n$, and $\bar\phi$ be defined by \eqref{varphi}.
	Then there exists a $\kappa>0$ such that
		\begin{gather*}%\label{C6.7-0}
			\max\{\vertiii{(x_1,\ldots,x_{n-1})}_{\bar\phi}, \vertiii{(x_n,\ldots,x_{n})}_{\bar\phi},\|x_n\|\}\le \kappa\cdot\vertiii{x}_\psi
		\end{gather*}	
	for all $x:=(x_1,\ldots,x_n)\in X^n$.
\end{proposition}	

\begin{proof}
	Let $x:=(x_1,\ldots,x_n)\in X^n$.
	By Proposition~\ref{P6.5}\;\eqref{P6.5-2}, $\|x_n\|\le\vertiii{x}_{\psi}$.
	By Proposition~\ref{P6.6}, $\bar\phi\in\pmb{\Psi}_{n-1}$.
	Thanks to  \eqref{A2}, $\vertiii{(x_n,\ldots,x_{n})}_{\bar\phi}\le (n-1)\|x_n\|\le (n-1)\cdot\vertiii{x}_{\psi}$.
	By  \eqref{P5.4-1} and Proposition~\ref{P7.1}\;\eqref{P7.1-6}, there exists a $M>0$ such that  $\vertiii{(x_1,\ldots,x_{n-1})}_{\bar\phi}\le\vertiii{\cdot}_\phi\le M\cdot\vertiii{x}_{\psi}$.
	Thus, the desired inequality holds with $\kappa:=\max\{M,n-1\}$.
\end{proof}

\begin{remark}\label{R8}
	Proposition~\ref{C6.5} offers flexibility in selecting norms when proving and formulating generalized separation results for collections of sets.
	A natural approach is to choose norms $\|\cdot\|_{\psi}$ and $\vertiii{\cdot}_\phi$ on $X^n$  associated with arbitrary functions $\psi,\phi\in\pmb{\Psi}_n$ and
	then define the norm $\vertiii{\cdot}_{\bar\phi}$ on $X^{n-1}$ with $\bar\phi$ given by \eqref{varphi}.
	This construction ensures that the norms $\|\cdot\|$, $\|\cdot\|_{\bar\phi}$ and $\vertiii{\cdot}_\psi$ satisfy all compatibility conditions \ref{C1}-\ref{C7}.
\end{remark}

\begin{theorem}[Extremal principle]\label{T6.8}
	Let $\psi\in\pmb{\Psi}_n$, $(X,\|\cdot\|)$ be Asplund, and $\Omega_1,\ldots,\Omega_n$  be closed.
	If $\{\Omega_1,\ldots,\Omega_n\}$ is extremal at $\bx$ with respect to $\vertiii{\cdot}_\psi$, then,
	for any $\varepsilon>0$, there exist $x_i\in\Omega_i\cap B_\eps(\bx)$ and  $x_i^*\in N^F_{\Omega_i}(x_i)$ $(i=1,\ldots,n)$ such that
	$\|\sum_{i=1}^nx_i^*\|<\eps$ and $
	\vertiii{(x_1^*,\ldots,x_n^*)}_{\psi^*}=1.
	$
\end{theorem}

\begin{proof}
	The statement is a consequence of Theorem~\ref{C6.4}.	
	By the second inequality of Proposition~\ref{L3.2}\;\eqref{L3.2-2} and the closedness of $\Omega_1,\ldots,\Omega_n$, the set  $\widehat\Omega$ is closed.
	Observe that $(X,\vertiii{\cdot}_\psi)$ is Asplund.
	Indeed, since $(X,\|\cdot\|)$ is Banach, it follows from the second inequality of Proposition~\ref{L3.2}\;\eqref{L3.2-2}  that $(X,\vertiii{\cdot}_\psi)$  is also Banach, while the Asplund propery stems from  the fact that the class of Asplund spaces is stable under Cartesian products; see \cite[p. 196]{Mor06.1}.
	By Proposition~\ref{P6.5}\;\eqref{P6.5-3} and \eqref{P6.5-4}, $N^F_{\widehat\Omega}(x)= N^F_{\Omega_1}(x_1)\times\ldots\times N^F_{\Omega_n}(x_n)$ for all $x:=(x_1,\ldots,x_n)\in\widehat\Omega$; see \cite{CuoKru24-2} for a detailed proof. 
	The first inequality in Proposition~\ref{L3.2}\;\eqref{L3.2-2} implies that $x_i\in B_\varepsilon(\bx)$ $(i=1,\ldots,n)$ whenever $(x_1,\ldots,x_n)\in B_\varepsilon(\bx,\ldots,\bx)$.
	This completes the proof.
\end{proof}	

\begin{remark}
	By Theorem~\ref{T6.8}, the extremality implies the existence of $n$ dual vectors satisfying the normalization condition $\vertiii{(x_1^*,\ldots,x_n^*)}_{\psi^*}=1$ and being normal (in the sense of Fr\'echet) to the respective sets at some points close (up to $\varepsilon$) to $\bx$, with their sum being almost (up to $\varepsilon$) zero.
\end{remark}	

 \section{The von Neumann-Jordan constant}\label{S7}
 Let $\|\cdot\|$ be a norm on a vector space $X$.
 Consider the function $\mathcal{G}_{\|\cdot\|}:X\times X\to\R$ defined by
 	\begin{gather}\label{S6.1}
 		\mathcal{G}_{\|\cdot\|}(x,y):=\dfrac{\|x+y\|^2+\|x-y\|^2}{2(\|x\|^2+\|y\|^2)}\;\text{for all}\;
 		x,y\in X\;\text{with}\; \|x\|+\|y\|>0.
 	\end{gather}	
 The von Neumann-Jordan constant ($C_{\text{NJ}}$) \cite{JorNeu35} of the norm $\|\cdot\|$ is defined by
 \begin{gather}\label{S7.1-2}
 	C_{\text{NJ}}(\|\cdot\|)=\sup_{\substack{x,y\in X,\|x\|+\|y\|>0} }\mathcal{G}_{\|\cdot\|}(x,y).
 \end{gather}	
 \begin{remark}\label{R16}
 	\begin{enumerate}[\rm (i)]
 		\item\label{R16-3}
 		By \eqref{S6.1}, $\mathcal{G}_{\|\cdot\|}(x+y,x-y)=\dfrac{1}{\mathcal{G}_{\|\cdot\|}(x,y)}$
 		for all $x,y\in X$ with $\mathcal{G}(x,y)>0$.	
 		\item\label{R16-1}
 		It is known that
 		$1\le C_{\text{NJ}}(\|\cdot\|)\le 2$
 		for every  Banach space $(X,\|\cdot\|)$; see, e.g., \cite{KatTak97}.
 		The completeness assumption can be omitted.
 		Indeed, condition $C_{\text{NJ}}(\|\cdot\|)\ge 1$ is a consequence of \eqref{S7.1-2} by letting $x=y$, while condition $C_{\text{NJ}}(\|\cdot\|)\le 2$ follows from the estimate
 		$\|x+y\|^2+\|x-y\|^2\le4(\|x\|^2+\|y\|^2)$ for all $x,y\in X$.
 		\item\label{R16-2}
 		$C_{\text{\rm NJ}}(X)=1$ if and only if $(X,\|\cdot\|)$ is a Hilbert space; cf.  \cite{KatTak97}.
 		In this case, we have $\|x+y\|^2+\|x-y\|^2=2(\|x\|^2+\|y\|^2)$ for all $x,y\in X$.
 	\end{enumerate}
 \end{remark}	
 
 \subsection{The $C_{\text{NJ}}$ for sign-symmetric norms}
 We study the von Neumann-Jordan constant of the primal and dual norms of the form  \eqref{T3.10-1}.
 Hereafter, we assume that $\psi,\phi\in \pmb{\Psi}_n$, and the constants $m$ and $M$, $m^*$ and $M^*$  are given by \eqref{mM}.
 \begin{theorem}\label{T7.1}
 	The following statements hold:
 	\begin{enumerate}[\rm (i)]
 		\item\label{T7.1-1}
 		$\frac{m^2}{M^2}\cdot C_{\text{\rm NJ}}(\vertiii{\cdot}_\phi) \le C_{\text{\rm NJ}}(\vertiii{\cdot}_\psi)\le \frac{M^2}{m^2}\cdot C_{\text{\rm NJ}}(\vertiii{\cdot}_\phi);$
 		\item\label{T7.1-2}
 		$\frac{{m}^2}{{M}^2}\cdot C_{\text{\rm NJ}}(\vertiii{\cdot}_{\phi^*}) \le C_{\text{\rm NJ}}(\vertiii{\cdot}_{\psi^*})\le\frac{{M}^2}{{m}^2}\cdot C_{\text{\rm NJ}}(\vertiii{\cdot}_{\phi^*}).$
 	\end{enumerate}
 \end{theorem}	
 
 \begin{proof}
 	Let $x,y\in X^n$ satisfy $\vertiii{x}_\psi+\vertiii{y}_\psi>0$.
 	By the first inequality in Proposition~\ref{P7.1}\;\eqref{P7.1-6}, 
 	\begin{align*}
 		\vertiii{x+y}^2_\phi+\vertiii{x-y}^2_\phi
 		&\le\dfrac{1}{m^2}\cdot\left(\vertiii{x+y}^2_\psi+\vertiii{x-y}^2_\psi\right)
 		\le \dfrac{1}{m^2}\cdot C_{\text{\rm NJ}}(\vertiii{\cdot}_\psi)\cdot 2(\vertiii{x}^2_\psi+\vertiii{y}^2_\psi) \\
 		&\le\dfrac{M^2}{m^2}\cdot C_{\text{\rm NJ}}(\vertiii{\cdot}_\psi)\cdot 2 \left(\vertiii{x}^2_\phi+\vertiii{y}^2_\phi\right).
 	\end{align*}	
 	Thus, $\frac{m^2}{M^2}\cdot C_{\text{\rm NJ}}(\vertiii{\cdot}_\phi) \le C_{\text{\rm NJ}}(\vertiii{\cdot}_\psi)$.
 	Similarly, applying the second inequality in Proposition~\ref{P7.1}\;\eqref{P7.1-6}, we obtain $ C_{\text{\rm NJ}}(\vertiii{\cdot}_\psi)\le \frac{M^2}{m^2}\cdot C_{\text{\rm NJ}}(\vertiii{\cdot}_\phi)$.
 	Thus, assertion \eqref{T7.1-1} is proved.
 	
 	Applying Proposition~\ref{P7.1}\;\eqref{P7.1-7}  and using similar arguments as in the proof of part \eqref{T7.1-1}, we obtain 
 	$\frac{{m^*}^2}{{M^*}^2}\cdot C_{\text{\rm NJ}}(\vertiii{\cdot}_{\phi^*}) \le C_{\text{\rm NJ}}(\vertiii{\cdot}_{\psi^*})\le\frac{{M^*}^2}{{m^*}^2}\cdot C_{\text{\rm NJ}}(\vertiii{\cdot}_{\phi^*})$.
 	The desired inequalities in assertion \eqref{T7.1-2} then follow from
 	Proposition~\ref{P7.1}\;\eqref{P7.1-3}.
 \end{proof}
 
 \begin{corollary}\label{C7.6}
 	Let $C_{\text{\rm NJ}}(\vertiii{\cdot}_\phi)=C_{\text{\rm NJ}}(\vertiii{\cdot}_{\phi^*})=1$.	
 	The following statements hold:
 	\begin{enumerate}[\rm (i)]
 		\item\label{C7.6-1}
 		if $\psi\ge\phi$, then $C_{\text{\rm NJ}}(\vertiii{\cdot}_\psi)\le M^2$
 		and $C_{\text{\rm NJ}}(\vertiii{\cdot}_{\psi^*})\le M^2$;
 		\item\label{C7.6-2}
 		if $\psi\le\phi$, then $C_{\text{\rm NJ}}(\vertiii{\cdot}_\psi)\le\dfrac{1}{m^2}$ and $C_{\text{\rm NJ}}(\vertiii{\cdot}_{\psi^*})\le\dfrac{1}{m^2}$;
 		\item\label{C7.6-3}
 		if $n=2$, then the inequalities in assertions \eqref{C7.6-1} and \eqref{C7.6-2} hold as equalities.
 	\end{enumerate}	
 \end{corollary}
 
 \begin{proof}	
 	Let $\psi\ge\phi$.
 	Proposition~\ref{P7.1}\;\eqref{P7.1-3} and \eqref{P7.1-4} imply that $m^*\cdot M=1$ and $M^*\le 1\le m$.
 	By Theorem~\ref{T7.1}, $C_{\text{\rm NJ}}(\vertiii{\cdot}_\psi)\le\frac{M^2}{m^2}\le M^2$ and 
 	$C_{\text{\rm NJ}}(\vertiii{\cdot}_{\psi^*})=\frac{{M^*}^2}{{m^*}^2}\le\frac{1}{{m^*}^2}=M^2$.
 	Thus, \eqref{C7.6-1} is proved.
 	The proof of \eqref{C7.6-2} is similar.
 	We now prove \eqref{C7.6-3}.
 	Suppose that $n=2$.
 	Let $\bar t:=(\bar t_1,\bar t_2)\in\Omega_2$ satisfy $M=\frac{\psi(\bar t)}{\phi(\bar t)}$,
 	$\bx:=(\bar t_1 \mathbf{u},0_X)$ and 
 	$\bar y:=(0_X,\bar t_2\mathbf{u})$ for some $\mathbf{u}\in\mathbb{S}_X$.	
 	Since $C_{\text{\rm NJ}}(\vertiii{\cdot}_\phi)=1$, we have
 	$\vertiii{\bx+\by}^2_\phi+\vertiii{\bx-\by}^2_\phi=2(\vertiii{\bx}^2_\phi+\vertiii{\by}^2_\phi).$
 	By \eqref{T3.10-1},
 	\begin{align*}
 		\vertiii{\bx+\by}^2_\psi+\vertiii{\bx-\by}^2_\psi
 		&=2\psi^2(\bar t)
 		=2M^2\phi^2(\bar t)
 		=M^2\left(\vertiii{\bx+\by}^2_\phi+\vertiii{\bx-\by}^2_\phi\right)\\
 		&=2M^2(\vertiii{\bx}^2_\phi+\vertiii{\by}^2_\phi)
 		=2M^2(\vertiii{\bx}^2_\psi+\vertiii{\by}^2_\psi).
 	\end{align*}	
 	Thus, $C_{\text{\rm NJ}}(\vertiii{\cdot}_\psi)\ge M^2$.
 	Let $\bar s:=(\bar s_1,\bar s_2)\in\Omega_2$ satisfy $m^*=\frac{\psi^*(\bar s)}{\phi^*(\bar s)}$,
 	$\bx^*:=(\bar s_1 \mathbf{u}^*,0_{X^*})$ and 
 	$\bar y^*:=(0_{X^*},\bar s_2\mathbf{u}^*)$ for some $\mathbf{u}^*\in\mathbb{S}_{X^*}$.	
 	Since $C_{\text{\rm NJ}}(\vertiii{\cdot}_{\phi^*})=1$, we have
 	$\vertiii{\bx^*+\by^*}^2_{\phi^*}+\vertiii{\bx^*-\by^*}^2_{\phi^*}=2(\vertiii{\bx^*}^2_{\phi^*}+\vertiii{\by^*}^2_{\phi^*}).$
 	By \eqref{T3.10-1},
 	\begin{align*}
 		2(\vertiii{\bx^*}^2_{\psi^*}+\vertiii{\by^*}^2_{\psi^*})
 		&=2(\vertiii{\bx^*}^2_{\phi^*}+\vertiii{\by^*}^2_{\phi^*})
 		=\vertiii{\bx^*+\by^*}^2_{\phi^*}+\vertiii{\bx^*-\by^*}^2_{\phi^*}\\
 		&=2{\phi^*}^2(\bar s)=2M^2{\psi^*}^2(\bar s)
 		=M^2\left(\vertiii{\bx^*+\by^*}^2_{\psi^*}+\vertiii{\bx^*-\by^*}^2_{\psi^*}\right).
 	\end{align*}	
 	Thus,
 	$C_{\text{\rm NJ}}(\vertiii{\cdot}_{\psi^*})\ge M^2$.
 	The inequalities $C_{\text{\rm NJ}}(\vertiii{\cdot}_\psi)\ge \dfrac{1}{m^2}$ and $C_{\text{\rm NJ}}(\vertiii{\cdot}_{\psi^*})\ge \dfrac{1}{m^2}$ are proved similarly.
 	The proof is complete.
 \end{proof}	
 
 \begin{remark}
 	When $X:=\mathbb{C}$, the primal norm estimates in Corollary~\ref{C7.6} recaptures \cite[Remark~5.3]{SaiKatTak00} and \cite[Theorems~2.3 \& 2.4]{Cao03}.
 	The assertions in part \eqref{C7.6-3} of Corollary~\ref{C7.3} may fail when $n\ge 3$; see
 	\cite[p. 904]{SaiKatTak00} for a counterexample.
 \end{remark}
 
 \subsection{A result of Clarkson revisited} We extend Clarkson's result  \cite{Cla37} from Lebesgue spaces to general normed vector spaces.
 
 \begin{definition}\label{D7.4}
 	A norm  $\|\cdot\|$ on a vector space $X$ is said to satisfy the \textit{Clarkson inequality} if
 		\begin{gather*}%\label{D7.4-1}
 			\|x+y\|^\be+\|x-y\|^\be\le 2(\|x\|^\al+\|y\|^\al)^{\frac{\be}{\al}}	
 		\end{gather*}	
 	for all $x,y\in X$,  $\al\in(1,2]$ and $\be\in\R$ with $\frac{1}{\al}+\frac{1}{\be}=1$. 
 \end{definition}	
 
 \begin{remark}\label{R20}
 	\begin{enumerate}[\rm (i)]
 		\item\label{R20-1}
 		In \cite[Theorem~2]{Cla36}, Clarkson proved that
 		$p$-norms on Lebesgue spaces  always satisfy the property in Definition~\ref{D7.4}.
 		To obtain this result, it is sufficient to prove that
 		$\abs{x+y}^\be+\abs{x-y}^\be\le 2(\abs{x}^\al+\abs{y}^\al)^{\frac{\be}{\al}}$
 		for all $x,y\in\mathbb{C}$, $\al\in(1,2]$ and $\be\in\R$ with $\frac{1}{\al}+\frac{1}{\be}=1$.
 		The result is then used to study the uniform convexity  \cite{Cla36} and the von  Neumann-Jordan constant \cite[Theorem, p. 114]{Cla37} of Lebesgue spaces.
 		\item
 		Clarkson's inequalities in general Banach spaces were first studied by Milman \cite{Mil84}.
 		It was later shown by Kato and Takahashi \cite{KatTak97} that a primal norm on a Banach space satisfies the \textit{Clarkson inequality} if and only if its dual  satisfies the inequality.
 		The reader is referred to \cite{PerTakKat00} and the references therein for discussions and implications of Clarkson-type inequalities.
 	\end{enumerate}
 \end{remark}	
 
 The next theorem  provides an explicit formula for the von Neumann-Jordan constant of a $p$-norm on general vector spaces.
 \begin{theorem}\label{T7.2}
 	Let $X$ be a normed space, and $\vertiii{\cdot}_p$ be defined by \eqref{pnorm} for some $p\in[1,\infty]$.
 	If the norm $\vertiii{\cdot}_p$ satisfies the Clarkson inequality, then
 		\begin{equation*}
 			C_{\text{\rm NJ}}(\vertiii{\cdot}_p)= 
 			\begin{cases}
 				2^{\frac{2-p}{p}}& \text{if } p\in[1,2],\\
 				2^{\frac{p-2}{p}}& \text{if } p\in(2,\infty].
 			\end{cases}
 		\end{equation*}
 \end{theorem}
 
 \begin{proof}
 	Let the norm $\vertiii{\cdot}_p$ satisfy the Clarkson inequality, $\bar x:=(\mathbf{u},0_X,\ldots,0_X)$ and $\bar y:=(0_X,\mathbf{u},0_X,\ldots,0_X)\in X^n$ for some $\mathbf{u}\in\mathbb{S}_X$.
 	By \eqref{S6.1} and Remark~\ref{R16}\;\eqref{R16-3},
 	\begin{gather}\label{T7.3-1}
 		\mathcal{G}_{\vertiii{\cdot}_p}(\bar x,\bar y)=2^{\frac{2-p}{p}},\;\;
 		\mathcal{G}_{\vertiii{\cdot}_p}(\bar x+\bar y,\bar x-\bar y)=\dfrac{1}{\mathcal{G}_{\vertiii{\cdot}_p}(\bar x,\bar y)}=2^{\frac{p-2}{p}}.
 	\end{gather}	
 	We consider two cases.\\
 	\textbf{Case 1}: $p\in[1,2]$. 
 	By the first equality in \eqref{T7.3-1},  $C_{\text{NJ}}(\vertiii{\cdot}_p)\ge 2^{\frac{2-p}{p}}$.
 	Thus, $C_{\text{NJ}}(\vertiii{\cdot}_{1})=2$.
 	It suffices to prove $C_{\text{NJ}}(\vertiii{\cdot}_p)\le 2^{\frac{2-p}{p}}$ for  $p\in(1,2]$.
 	Set $q:=\frac{p}{p-1}\ge 2$.
 	Let $x,y\in X^n$ with $\vertiii{x}_p+\vertiii{y}_p>0$.
 	By the Clarkson inequality with $\al:=p$ and  $\be:=q$, we obtain
 	$
 	\vertiii{x+y}^q_p+\vertiii{x-y}^q_p\le 
 	2(\vertiii{x}^p_p+\vertiii{y}^p_p)^{\frac{q}{p}}.
 	$
 	From this and the H\"older inequality,
 	\begin{align*}
 		\vertiii{x+y}^2_p+\vertiii{x-y}^2_p
 		&\le2^{\frac{q-2}{q}}\cdot(\vertiii{x+y}^q_p+\vertiii{x-y}^q_p)^{\frac{2}{q}}
 		\le2\left(\vertiii{x}^p_p+\vertiii{y}^p_p\right)^{\frac{2}{p}}\\
 		&\le 2\left[2^{\frac{2-p}{2}}(\vertiii{x}^2_p+\vertiii{y}^2_p)^{\frac{p}{2}}\right]^{\frac{2}{p}}
 		=2(\vertiii{x}^2_p+\vertiii{y}^2_\psi)\cdot2^{\frac{2-p}{p}}.
 	\end{align*}	
 	Thus, $C_{\text{NJ}}(\vertiii{\cdot}_p)\le 2^{\frac{2-p}{p}}$.\\
 	\textbf{Case 2}: $p\in(2,\infty]$. 
 	By the second equality in \eqref{T7.3-1}, $C_{\text{NJ}}(\vertiii{\cdot}_p)\ge 2^{\frac{p-2}{p}}$.
 	Thus, $C_{\text{NJ}}(\vertiii{\cdot}_{\infty})=2$.
 	It suffices to show that $C_{\text{NJ}}(\vertiii{\cdot}_p)\le 2^{\frac{2-p}{p}}$ for  $p\in(2,\infty)$.
 	Set $q:=\frac{p}{p-1}\in(1,2)$.
 	Using similar arguments as in \textbf{Case 1} with $p$ and $q$ interchanged, we obtain 
 	$C_{\text{NJ}}(\vertiii{\cdot}_p)\le 2^{\frac{2-q}{q}}= 2^{\frac{2-p}{p}}$.
 \end{proof}	
 
 \begin{corollary}\label{C7.3}
 	Let $X$ be a normed space, $\vertiii{\cdot}_p$ be defined by \eqref{pnorm} for some $p\in[1,\infty]$, and
 	$q\in[1,\infty]$ satisfy $\frac{1}{p}+\frac{1}{q}=1$.
 	If the norm $\vertiii{\cdot}_p$ satisfies the Clarkson inequality, then
 		\begin{gather*}
 			C_{\text{\rm NJ}}(\vertiii{\cdot}_p)=C_{\text{\rm NJ}}(\vertiii{\cdot}_{q})=2^{\frac{2}{\min\{p,q\}}-1}.
 		\end{gather*}
 \end{corollary}
 \begin{proof}
 	Let the norm $\vertiii{\cdot}_p$ satisfy  the Clarkson inequality, and $q\in[1,\infty]$ with $\frac{1}{p}+\frac{1}{q}=1$.
 	We use Theorem~\ref{T7.2} in the remainder of the proof.
 	Observe that $C_{\text{\rm NJ}}(\vertiii{\cdot}_{1})=C_{\text{\rm NJ}}(\vertiii{\cdot}_{\infty})=2=2^{\frac{1}{\min\{1,\infty\}}-1}$.
 	If $p\in(1,2]$, then $q=\frac{p}{p-1}\ge 2$.
 	In this case, $p=\min\{p,q\}$ and $C_{\text{\rm NJ}}(\vertiii{\cdot}_{q})=2^{1-\frac{2}{q}}=2^{\frac{2}{p}-1}=C_{\text{\rm NJ}}(\vertiii{\cdot}_p)$.
 	If $p\in(2,\infty)$, then $q=\frac{p}{p-1}\in(1,2)$.
 	In this case, $q=\min\{p,q\}$ and $C_{\text{\rm NJ}}(\vertiii{\cdot}_{q})=2^{\frac{2}{q}-1}=2^{1-\frac{2}{p}}=C_{\text{\rm NJ}}(\vertiii{\cdot}_p).$
 	The proof is complete.
 \end{proof}	
 \begin{remark}
 	When $X$ is a Lebesgue space, Corollary~\ref{C7.3} recaptures \cite[Theorem, p. 114]{Cla37}. 
 	In this case,  $C_{\text{\rm NJ}}(\vertiii{\cdot}_{2})=1$.
 \end{remark}	

\section{Conclusions}\label{S8}
A class of sign-symmetric (strictly convex) norms on product spaces is studied.
Correspondences between a class of  sign-symmetric norms (strictly convex)  and 
a family of (strictly) convex continuous functions  are established.
Explicit formulas for the dual norm and the convex subdifferential of a primal sign-symmetric norm are provided.
It is demonstrated that these norms are well-suited for formulating and proving the extremal principle and other generalized separation results.
As an application, the von Neumann–Jordan constant of norms on product spaces is investigated, and a classical result of Clarkson is revisited.

The following topics will be studied in the future.
\begin{enumerate}[\rm (i)] 
	\item
	Study uniform convexity relations between the class of sign-symmetric norms and the corresponding class of convex continuous functions.
	\item 
	Corollary~\ref{C7.3} shows that the von Neumann–Jordan constants of a $p$-norm and its dual are equal (under the Clarkson inequality assumption).
	By Corollary~\ref{C7.6}\;\eqref{C7.6-3}, the constants of a given norm of the form \eqref{T3.10-1} and its dual are also equal when $n = 2$. 
	Does the equality hold for any norm of the form \eqref{T3.10-1}  and its dual when $n\ge 3$?
	\item
	Study sign-symmetric norms on product spaces where the component normed spaces are distinct.
\end{enumerate}

\section*{Acknowledgments}
	The author wishes to thank Professor Alexander Kruger for  his comments and the anonymous referee for suggestions which helped us improve the manuscript.

\section*{Funding}
This work is supported by the Ministry of Education and Training of Vietnam under Grant No. B2024-CTT-03.

\section*{Disclosure statement}
The author reports there are no competing interests to declare.

\section*{Data availability statement}
Data sharing is not applicable to this article as no new data were created or analyzed in this study.

\section*{ORCID}
Nguyen Duy Cuong http://orcid.org/0000-0003-2579-3601


\begin{thebibliography}{10}
	\providecommand{\url}[1]{{#1}}
	\providecommand{\urlprefix}{URL }
	\expandafter\ifx\csname urlstyle\endcsname\relax
	\providecommand{\doi}[1]{DOI~\discretionary{}{}{}#1}\else
	\providecommand{\doi}{DOI~\discretionary{}{}{}\begingroup
		\urlstyle{rm}\Url}\fi
	
\bibitem{BonDun73}
Bonsall, F.F., Duncan, J.: 	Numerical Ranges II, 
London Mathematical Society Lecture Note Series, Cambridge University Press (1973).
\newblock DOI: 10.1017/CBO9780511662515

\bibitem{Cao03}
Cao, H.X.: Von Neumann-Jordan constants of absolute normalized norms on $\mathbb{C}^n$. \newblock Acta Math. Sin. \textbf{19}(3), 507--512 (2003).
\newblock DOI: 10.1007/s10114-003-0272-4

\bibitem{Cla36}
Clarkson, J.A.: Uniformly convex spaces.
\newblock Trans. Amer. Math. Soc. \textbf{40}(3), 396--414 (1936).
\newblock DOI: 10.1090/S0002-9947-1936-1501880-4

\bibitem{Cla37}
Clarkson, J.A.: The von Neumann-Jordan constant for the Lebesgue spaces.
\newblock Ann. Math. \textbf{38}, 114--115 (1937).
\newblock DOI: 10.2307/1968512

\bibitem{CuoKru24-2}
Cuong, N. D.,  Kruger, A. Y. (2025). Generalized separation of collections of sets. Optimization, 1-23. DOI: https://doi.org/10.1080/02331934.2025.2562434

\bibitem{SomAttSat05}
Dhompongsa, S., Kaewkhao, A., Saejung, S.: Uniform smoothness and $U$-convexity and  $\psi$-direct sums.
\newblock J. Nonlinear Convex Anal. \textbf{6}(2), 327--338 (2005)

\bibitem{DonRoc14} 
Dontchev, A.L., Rockafellar, R.T.: Implicit Functions and Solution Mappings. A View from Variational Analysis, 2 edn, Springer, New York (2014).
\newblock DOI: 10.1007/978-1-4939-1037-3 

\bibitem{Iof17} 
Ioffe, A.D.: Variational Analysis of Regular Mappings. Theory and Applications,  Springer (2017).
\newblock DOI: 10.1007/978-3-319-64277-2 

\bibitem{JorNeu35}
Jordan, P., Neumann, J.V.: On inner products in linear, metric spaces.
\newblock Ann. Math. \textbf{36}(3), 719--723 (1935).
\newblock DOI: 10.2307/1968653 

\bibitem{KatSaiTam03}
Kato, M., Saito, K.S., Tamura, T.: On $\psi$-direct sums of Banach spaces and convexity.
\newblock J. Aust. Math. Soc. \textbf{75}(3), 413--422 (2003).
\newblock DOI: 10.1017/s1446788700008193 

\bibitem{KatTak97}
Kato, M., Takahashi, Y.: On the von Neumann-Jordan constant for Banach spaces.
\newblock Proc. Amer. Math. Soc. \textbf{125}(4), 1055--1062 (1997).
\newblock DOI: 10.1090/s0002-9939-97-03740-4 

\bibitem{KruMor80}
Kruger, A.Y., Mordukhovich, B.S.: Extremal points and the Euler equation in nonsmooth optimization problems. 
\newblock Dokl. Akad. Nauk BSSR 24(8), 684--687 (1980). 
\newblock In Russian. Available from:
https://asterius.federation.edu.au/akruger/research/publications.html.
\bibitem{LooSte90} 
Loomis, L., Sternberg, S.: Advanced Calculus, Jones and Bartlett Publishers (1990)

\bibitem{Mil84}
Milman, M.: Complex interpolation and geometry of Banach spaces.  \newblock Ann. Mat. Pura Appl. \textbf{136}(1), 317--328 (1984). 
\newblock DOI: 10.1007/bf01773388 

\bibitem{MitOshSai05}
Mitani, K.I., Oshiro, S., Saito, K.S.: Smoothness of $\psi$-direct sum of Banach spaces.
\newblock Math. Inequal. Appl. \textbf{8}(1), 147--157 (2005).
\newblock DOI: 10.7153/mia-08-14  

\bibitem{MitSaiSuz03}
Mitani, K.I., Saito, K.S., Suzuki, T.: Smoothness of absolute norms on $\mathbb{C}^n$.
\newblock J. Convex Anal. \textbf{10}(1), 89--107 (2003)

\bibitem{MorNam22}
Mordukhovich, B., Nam, N.: Convex Analysis and Beyond. Volume I: Basic Theory, Springer (2022).
DOI: 10.1007/978-3-030-94785-9 

\bibitem{Mor06.1}
Mordukhovich, B.S.: Variational Analysis and Generalized Differentiation. I: Basic Theory, vol. 330 Springer, Berlin (2006).
\newblock DOI: 10.1007/3-540-31247-1

\bibitem{PerTakKat00}
Persson, L.E., Takahashi, Y., Kato, M.: Clarkson type inequalities and their relations to the concepts of type and cotype.
\newblock Collect. Math. \textbf{51}(3), 327--346 (2000)

\bibitem{RocWet98}
Rockafellar, R.T., Wets, R.J.B.: Variational Analysis, Springer, Berlin (1998).
\newblock DOI: 10.1007/978-3-642-02431-3 

\bibitem{SaiKatTak00}
Saito, K.S., Kato, M., Takahashi, Y.: Absolute norms on $\mathbb{C}^n$. 
\newblock J. Math. Anal. Appl. \textbf{252}(2), 879--905 (2000). 
\newblock DOI: 10.1006/jmaa.2000.7139
\end{thebibliography}
\end{document}